\documentclass[11pt, a4paper]{article}
\usepackage{amssymb,amsmath,amsthm,mathrsfs}
\usepackage{fancyhdr}
\usepackage[british]{babel}
\usepackage{enumitem}
\usepackage{lineno}
\usepackage{tikz}

\usepackage{hyperref}
\usepackage[margin=2.5cm]{geometry}

\newtheorem{theorem}{Theorem}[section]
\newtheorem{prop}[theorem]{Proposition}
\newtheorem{lemma}[theorem]{Lemma}

\theoremstyle{definition}
\newtheorem{problem}{Problem}

\newtheorem{defn}[theorem]{Definition}
\newtheorem{fac}{Fact}
\newtheorem{conjecture}[theorem]{Conjecture}
\newtheorem{coro}[theorem]{Corollary}
\newtheorem{rmk}{Remark}

\newtheorem{claim}[theorem]{Claim}

\newcommand{\eps}{\varepsilon}

\title{$H$-factors in graphs with small independence number}

\author{
	Ming Chen\thanks{School of Mathematics, China University of Mining and Technology, Xuzhou, China. Email: {\tt chenming314@cumt.edu.cn}.}
	\and
	Jie Han\thanks{School of Mathematics and Statistics and Center for Applied Mathematics, Beijing Institute of Technology, Beijing, China. Email: {\tt hanjie@bit.edu.cn}.}
	\and
	Guanghui Wang\thanks{School of Mathematics and Data Science Institute, Shandong University, Jinan, China. Email: {\tt ghwang@sdu.edu.cn}. Supported by Young Taishan Scholar program of Shandong Province (201909001) and Natural Science Foundation of China (11871311).}
	\and
	Donglei Yang\thanks{Data Science Institute, Shandong University, Jinan, China. Email: {\tt dlyang@sdu.edu.cn}. Supported by the China Postdoctoral Science Foundation (2021T140413), Natural Science Foundation of China (12101365) and Natural Science Foundation of Shandong Province (ZR2021QA029).}
}

\begin{document}
\maketitle

\begin{abstract}
Let $H$ be an $h$-vertex graph. The vertex arboricity $ar(H)$ of $H$ is the least integer $r$ such that $V(H)$ can be partitioned into $r$ parts and each part induces a forest in $H$. We show that for sufficiently large $n\in h\mathbb{N}$, every $n$-vertex graph $G$ with $\delta(G)\geq \left(\max\left\{1-\frac{2}{f(H)},\frac{1}{2}\right\}+o(1)\right)n$ and $\alpha(G)=o(n)$ contains an $H$-factor, where $f(H)=2ar(H)$ or $2ar(H)-1$. The result can be viewed an analogue of the Alon--Yuster theorem \cite{MR1376050} in Ramsey--Tur\'{a}n theory, which generalises the results of Balogh--Molla--Sharifzadeh~\cite{MR3570984} and Knierm--Su~\cite{MR4193066} on clique factors. In particular the degree conditions are asymptotically sharp for infinitely many graphs $H$ which are not cliques.
\end{abstract}

\section{Introduction}
Let $H$ be an $h$-vertex graph and $G$ be an $n$-vertex graph. An $H$-\emph{tiling} in $G$ is a collection of vertex-disjoint subgraphs of $G$ isomorphic to $H$. An $H$-\emph{factor} is an $H$-tiling which covers all the vertices of $G$. Determining sufficient conditions for the existence of an $H$-factor is one of the fundamental lines of research in extremal graph theory. One important reason is due to a result of Hell and Kirkpatrick \cite{1983Hell} which shows that the decision problem for $H$-factors is \emph{NP}-complete, given that $H$ has a connected component of size at least 3.

The first result in this line is due to Dirac \cite{Dirac1952}, who proved that every $n$-vertex graph $G$ with $\delta(G)\geq\frac{n}{2}$ contains a Hamiltonian cycle, in particular if $n$ is even then $G$ has a perfect matching. The celebrated Hajnal--Szemer\'{e}di theorem \cite{HajnalSze1970} states that for all integers $n, r$ with $r\geq 2$ and $r|n$, any $n$-vertex graph $G$ with $\delta(G)\geq \left(1-\frac{1}{r}\right)n$ contains a $K_{r}$-factor. The $K_{3}$-factor case was previously proved by Corr\'{a}di and Hajnal \cite{MR200185}. The balanced complete $r$-partite graph witnesses the tightness of the minimum degree condition.
For non-cliques $H$, Alon and Yuster \cite{MR1376050} proved that for any constant $\eps>0$ and large $n\in\mathbb{N}$ divisible by $|H|$, $\delta(G)\geq \left(1-\frac{1}{\chi(H)}+\eps\right)n$ guarantees an $H$-factor, namely, the relevant parameter in the degree condition is $\chi(H)$, the \emph{chromatic number of $H$}.
However, as given in \cite{cooley2007perfect,Kawa}, there are graphs $H$ for which the term $1-1/\chi(H)$ in the minimum degree condition can be improved significantly and determining the best possible bound on $\delta(G)$ for an arbitrary graph $H$ has been an intriguing problem.
There had been many excellent results (see e.g. \cite{MR1767021, MR1829855, MR1995690} and survey \cite{MR2588541}) until it was finally settled by K\"{u}hn and Osthus \cite{DKDO2009} in 2009. There are also several significant generalisations of the Hajnal--Szemer\'{e}di theorem in the setting of partite graphs \cite{MR3354296, MR1910115, MR2433861}, directed graphs \cite{MR3406450} and hypergraphs \cite{MR2500161}.

\subsection{Main result}

Note that the extremal example that achieves the optimality of the bound on $\delta(G)$ in the Hajnal--Szemer\'{e}di theorem contains a large independent set, which is ``regular'' and rather rare among all graphs. Following the spirit of the well-known Ramsey--Tur\'{a}n theory (see \cite{MR0299512,MR1476449, MR716422, MR2500161}), a natural question on the Hajnal--Szemer\'{e}di theorem is to determine the minimum degree condition forcing a clique factor when the host graph has sublinear independence number. The following Ramsey--Tur\'{a}n type problem was proposed by Balogh, Molla and Sharifzadeh \cite{MR3570984}.

\begin{problem}[\cite{MR3570984}]\label{pro1.1}
Let $r\geq 3$ be an integer and $G$ be an $n$-vertex graph with $\alpha(G)=o(n)$. What is the minimum degree condition on $G$ that guarantees a $K_{r}$-factor?
\end{problem}

Balogh, Molla and Sharifzadeh \cite{MR3570984} studied the $K_{3}$-factors and showed that the minimum degree condition $\delta(G)\ge\frac{n}{2}+\mu n$ guarantees a triangle factor, for any $\mu>0$ and large $n\in 3\mathbb N$.

Later, Nenadov and Pehova \cite{MR4080942} generalised Problem \ref{pro1.1} to $K_{\ell}$-independence numbers for $\ell\geq 2$: is it true that for every $r, \ell\in \mathbb{N}$ with $r\geq \ell\geq 2$ and $\mu>0$ there exist a constant $\alpha$ and $n_{0}\in \mathbb{N}$ such that every graph $G$ on $n\geq n_{0}$ vertices where $r$ divides $n$ with $\delta(G)\geq \max\left\{\frac{1}{2}+\mu, \frac{r-\ell}{r}+\mu\right\}n$ and $\alpha_{\ell}(G)\leq \alpha n$ has a $K_{r}$-factor? Nenadov and Pehova \cite{MR4080942} also proved a general upper bound for the minimum degree conditions, which is asymptotically optimal when $\ell=r-1$. Chang, Han, Kim, Wang and Yang \cite{CHKWY} provided a negative answer to this problem. They \cite{CHKWY} also determined an asymptotically optimal degree condition for $\ell\geq \frac{3}{4}r$.

Recently, Knierim and Su \cite{MR4193066} proved the following result, which solves Problem \ref{pro1.1} asymptotically.
\begin{theorem}[\cite{MR4193066}]\label{thm1.31}
Given constant $\mu >0$, $r\in \mathbb{N}$ with $r\geq 4$, there exists $\alpha >0$ such that the following holds for sufficiently large $n$. Let $G$ be an $n$-vertex graph with $r$ dividing $n$, $\delta(G)\geq \left(1-\frac{2}{r}+\mu\right)n$ and $\alpha(G)\leq \alpha n$. Then $G$ contains a $K_{r}$-factor.
\end{theorem}

It is natural to study Problem~\ref{pro1.1} for arbitrary $H$ other than cliques.
The first attempt would be to prove a result that matches the Alon--Yuster theorem \cite{MR1376050}, which we shall formulate below.
For our problem, we find that the relevant parameter for the minimum degree condition is the vertex arboricity of $H$.
\begin{defn}[]\label{def1.41}
The \emph{vertex arboricity $ar(H)$ of a graph $H$} is the minimum integer $r$ such that the vertices of $H$ can be partitioned into $r$ subsets each of which induces a forest. The corresponding partition is called an \emph{acyclic partition} of $H$. We say that an acyclic partition of $H$ is \emph{optimal} if it has exactly $ar(H)$ parts.
\end{defn}

Note that this definition extends that of the chromatic number of $H$, where each color class is required to be an independent set.
Our first result provides an analogue of the Alon--Yuster theorem \cite{MR1376050} in host graphs with sublinear independence number.

\begin{theorem}\label{coro1.311}
Given $\mu >0$, $h\in \mathbb{N}$ with $h\geq 3$ and an $h$-vertex graph $H$, there exists $\alpha >0$ such that the following holds for sufficiently large $n\in h\mathbb{N}$. Let $G$ be an $n$-vertex graph such that $\delta(G)\geq \max\left\{\left(1-\frac{1}{ar(H)}+\mu\right)n, \left(\frac{1}{2}+\mu\right)n\right\}$ and $\alpha(G)\leq \alpha n$. Then $G$ contains an $H$-factor.
\end{theorem}


Similar to the Alon--Yuster theorem \cite{MR1376050}, the minimum degree condition in Theorem~\ref{coro1.311} is asymptotically tight for infinitely many graphs $H$ which we define now.
Let $H$ be a graph and $\mathscr P$ be a family of vertex partitions of $H$ where each partition has $\ell$ parts.
For $\mathcal P\in \mathscr P$, let $x_{1}\leq x_{2}\leq\cdots \leq x_{\ell}$ denote the sizes of the parts of $\mathcal P$.
Put $\mathcal{D}(\mathcal P):=\{x_{i+1}-x_{i}\mid i=1, \dots, \ell-1\}$.
Let $\mathcal{D}(H, \mathscr P)$ denote the union of all the sets $\mathcal{D}(\mathcal P)$ taken over all $\mathcal P\in \mathscr P$.
Let $\mathrm{hcf}_{1}(H, \mathscr P)$ be the highest common factor of all integers in $\mathcal{D}(H, \mathscr P)$.
(If $\mathcal{D}(H, \mathscr P)=\emptyset$, then we set $\mathrm{hcf}_{1}(H, \mathscr P):=\infty$.)
Let $\mathrm{hcf}_{2}(H, \mathscr P)$ be the highest common factor of all the orders of components of $H$.
If $\ell=1$ and $\mathrm{hcf}_{2}(H, \mathscr P)=1$, then $\mathrm{hcf}(H, \mathscr P)=1$.
If $\ell=2$, $\mathrm{hcf}_{2}(H, \mathscr P)=1$ and $\mathrm{hcf}_{1}(H, \mathscr P)\leq 2$, then $\mathrm{hcf}(H, \mathscr P)=1$.
If $\ell \geq3$ and $\mathrm{hcf}_{1}(H, \mathscr P)=1$, then $\mathrm{hcf}(H, \mathscr P)=1$.

Given a graph $H$, let $\text{PC}$ be the family of all proper colorings of $H$ with $\chi(H)$ colors, and let $\text{AP}$ be the family of all acyclic partitions of $H$ with $ar(H)$ parts.
Then for $\ell\geq 2$, whether $\mathrm{hcf}(H, \text{PC})=1$ or not is the dichotomy found by K\"{u}hn and Osthus \cite{DKDO2009} for the minimum degree conditions for $H$-factors in general graphs.
That is, the term $1-\frac{1}{\chi(H)}$ in the Alon--Yuster theorem is asymptotically tight if and only if $\mathrm{hcf}(H, \text{PC})\neq 1$. In the problem of K\"{u}hn and Osthus \cite{DKDO2009}, $\ell=1$ implies that $H$ is an independent set, which is a trivial case. However, for our problem, $\ell=1$ means that $H$ is a forest, in which case we have the following result.

\begin{theorem}\label{thm1.2345}
Given $\mu >0$, $h\in \mathbb{N}$ with $h\geq 3$ and an $h$-vertex forest $H$ with $\mathrm{hcf}_{2}(H, \text{AP})=1$, there exists $\alpha >0$ such that the following holds for sufficiently large $n\in h\mathbb{N}$. Let $G$ be an $n$-vertex graph such that $\delta(G)\geq \mu n$ and $\alpha(G)\leq \alpha n$. Then $G$ contains an $H$-factor.
\end{theorem}

Note that the degree condition in Theorem \ref{thm1.2345} is quite small. Its proof follows the lattice-based absorbing method \cite{MR3632565} used in previous works as well as in this paper. The proof is omitted here and will appear in the first author's doctoral thesis.

For our problem, we show that the minimum degree condition in Theorem~\ref{coro1.311} is asymptotically tight for $H$ with $\mathrm{hcf}(H, \text{AP})\neq 1$ and conjecture that the relation is actually ``if and only if'' (see Conjecture~\ref{conj} below). The proof of the following Proposition \ref{prop5.111} is presented in full details in Section \ref{sec555}.

\begin{prop}[]\label{prop5.111}
Given $\alpha>0$ and $h\in \mathbb{N}$, let $H$ be an $h$-vertex graph with $\mathrm{hcf}(H, \text{AP})\neq 1$.
Then the following holds for sufficiently large $n$.
There exists an $n$-vertex graph $G$ with $\delta(G)\geq \max\{(1-\frac{1}{ar(H)})n-2, \frac{n}{2}-2\}$ and $\alpha(G)\leq \alpha n$ such that $G$ has no $H$-factor.
\end{prop}
Comparing with Theorem \ref{coro1.311}, we actually prove a slightly stronger result. To state it, we first introduce some notation.

\begin{defn}[]\label{def1.42}
Let $\widetilde{\mathcal{H}}$ be a family of graphs such that every element $H\in \widetilde{\mathcal{H}}$ has an acyclic partition $\mathcal{P}=\{T_{1},\ldots, T_{r}\}$ with $r:=ar(H)$ such that i) $H[T_{1}]$ is an independent set and ii) $2|T_{1}|=|T_{i}|$ for each $i\in [2, r]$.
%
Given a graph $H$, let $f(H)$ be an integer such that
\[
f(H) := \begin{cases} 2ar(H)-1 \quad \text{ if } H\in \widetilde{\mathcal{H}}; \\
2ar(H) \hfill \text{ otherwise}.
\end{cases}
\]
\end{defn}
Note that all cliques of odd order belong to $\widetilde{\mathcal{H}}$ and $f(K_{r})=r$. Our main result reads as follows.
\begin{theorem}\label{thm1.3}
Given $\mu >0$, $h\in \mathbb{N}$ and an $h$-vertex graph $H$, there exists $\alpha >0$ such that the following holds for sufficiently large $n\in h\mathbb{N}$. Let $G$ be an $n$-vertex graph such that $\delta(G)\geq \max\left\{\left(1-\frac{2}{f(H)}+\mu\right)n, \left(\frac{1}{2}+\mu\right)n\right\}$ and $\alpha(G)\leq \alpha n$. Then $G$ contains an $H$-factor.
\end{theorem}

Since $f(H)\le 2ar(H)$, Theorem~\ref{coro1.311} follows from Theorem \ref{thm1.3}.
Theorem \ref{thm1.3} unifies and generalises the results of Balogh--Molla--Sharifzadeh~\cite{MR3570984} and Knierm--Su~\cite{MR4193066} on clique factors.
We remark that the minimum degree condition in Theorem \ref{thm1.3} is asymptotically sharp for (all graphs $H$ with ${\mathrm{hcf}}(H, \text{AP})\neq 1$ and) infinitely many graphs $H$ with ${\mathrm{hcf}}(H, \text{AP})=1$.
To illustrate this we introduce the following notation.

\begin{defn}[]\label{def1.31}
The \emph{critical arboricity} $ar_{cr}(H)$ of a graph $H$ is defined as $\frac{\left(ar(H)-1\right)|V(H)|}{|V(H)|-\sigma(H)}$, where $\sigma(H)$ denotes the minimum size of a part taken over all optimal acyclic partitions of $H$.
\end{defn}

Note that $ar(H)-1 < ar_{cr}(H)\le ar(H)$.
We supply the following general lower bound serving as a ``space barrier'' for this problem.

\begin{prop}\label{prop1.8}
Let $h\in \mathbb{N}$ and $H$ be an $h$-vertex graph. Then for any $\alpha>0$, the following holds for sufficiently large $n\in h\mathbb{N}$. There exists an $n$-vertex graph $G$ with $\delta(G)\geq\left(1-\frac{1}{ar_{cr}(H)}\right)n-1$ and $\alpha(G)\leq \alpha n$ which does not have an $H$-factor.
\end{prop}

In particular, Proposition \ref{prop1.8} implies that the degree condition in Theorem \ref{thm1.3} is tight for the graphs $H$ with $f(H)=2ar_{cr}(H)$.
In particular, when $ar_{cr}(H)<ar(H)$, $f(H)=2ar_{cr}(H)$ implies that $H\in \widetilde{\mathcal{H}}$ and thus $\mathrm{hcf}(H, \text{AP})=1$.
The following is an example for such $H$.
Let $r\in \mathbb{N}$ and $H$ be an $(r+1)$-partite graph $K_{2r+1, \dots, 2r+1}$.
By definition we have $K_{2r+1, \dots, 2r+1}\in \widetilde{\mathcal{H}}, f(K_{2r+1, \dots, 2r+1})=2r+1$ and $ar_{cr}(K_{2r+1, \dots, 2r+1})=\frac{2r+1}{2}$. The proof of Proposition \ref{prop1.8} is presented in full details in Section \ref{sec555}.



In summary, the degree condition in Theorem \ref{thm1.3} is tight for the graphs $H$ with i) ${\mathrm{hcf}}(H, \text{AP})\\ \neq 1$ and ii) $H\in \widetilde{\mathcal{H}}$ and $ar_{cr}(H) = ar(H)-1/2$.
Nevertheless, we put forward the following conjecture for the $H$-factor problem.
\begin{conjecture}\label{conj}
Given $\mu >0$, $h\in \mathbb{N}$ and an $h$-vertex graph $H$ with ${\mathrm{hcf}}(H, \text{AP})= 1$, there exists $\alpha >0$ such that the following holds for sufficiently large $n\in h\mathbb{N}$. Let $G$ be an $n$-vertex graph such that $\delta(G)\geq \max\left\{\left(1-\frac{1}{ar_{cr}(H)}+\mu\right)n, \left(\frac{1}{2}+\mu\right)n\right\}$ and $\alpha(G)\leq \alpha n$. Then $G$ contains an $H$-factor.
\end{conjecture}


For complete multi-partite graphs $H$, the smallest open case is $H=K_{3,3,4}$ where $f(H)=5$, $ar_{cr}(H)=\frac{20}{9}$ and ${\mathrm{hcf}}(H, \text{AP})=1$.

\subsection{Proof strategy}

Our proof makes use of the absorption method and builds on the techniques developed in \cite{MR3529107, MR3632565, MR3290271}. The absorption method was introduced by R\"{o}dl, Ruci\'{n}ski and Szemer\'{e}di about a decade ago in \cite{MR2500161}. Since then, it has turned out to be an important tool for studying the existence of spanning structures in graphs, digraphs and hypergraphs.

Now we sketch the proof idea. The main tasks are to (i) build an absorbing set and (ii) cover almost all of the remaining vertices with an $H$-tiling. For (i), we need the following notation of absorbers and absorbing sets in \cite{MR4080942}.
\begin{defn}
  Let $G$ be an $n$-vertex graph and $H$ be an $h$-vertex graph. Then
\begin{itemize}
  \item[(1)] a subset $A\subseteq V(G)$ is a $\xi$-\emph{absorbing set} in $G$ for some constant $\xi$ if for any subset $U\subseteq V(G)\backslash A$ of size at most $\xi n$ and $|A\cup U|\in h\mathbb{N}$, $G[A\cup U]$ contains an $H$-factor.
  \item[(2)] for any subset $S\in V(G)$ of size $h$ and an integer $t$, we call a subset $A_{S}\subseteq V(G)\backslash S$ an $(H, t)$-\emph{absorber} for $S$ if $|A_{S}|=ht$ and both $G[A_{S}]$ and $G[A_{S}\cup S]$ contain an $H$-factor.
\end{itemize}
\end{defn}
Widely used constructions of absorbing set by R\"{o}dl, Ruci\'{n}ski and Szemer\'{e}di \cite{MR2500161} and H\`{a}n, Person and Schacht \cite{MR2496914} rely on the property that
\begin{quote}
  \emph{every $h$-subset $S$ has $\Omega(n^{ht})$ $(H, t)$-absorbers for $S$.}
\end{quote}
However, as pointed out in \cite{MR3570984}, in our setting this is usually impossible because when we construct the absorbers using the independence number condition, it does not give such a strong counting. Instead, a new approach due to Nenadov and Pehova \cite{MR4080942} guarantees an absorbing set provided that
\begin{quote}
\emph{every $h$-set $S$ has $\Omega(n)$ vertex-disjoint $(H, t)$-absorbers for $S$.}
\end{quote}
However, it is still unclear how to verify this for our problem and the bulk of the work on building absorbing sets is to handle this. Here we instead build a partition $V(G)=B\cup U$ satisfying that
\begin{quote}
\emph{$|B|=o(n)$ and every $h$-set in $U$ has $\Omega(n)$ vertex-disjoint absorbers.}
\end{quote}
Similar ideas already appeared in our recent work, e.g.~\cite{CHKWY}. Then the arguments reduce to finding an absorbing set $A$ in $G[U]$ and then covering vertices of $B$ by a small $H$-tiling vertex disjoint from $A$ so as to yield a desired absorbing set.
Our proof for this is also significantly more involved than that in \cite{MR4193066} and uses the lattice-based absorption method by the second author \cite{MR3632565}.

Our second task is to establish an almost perfect $H$-tiling. Following a standard application of the regularity lemma, we obtain an $\eps$-regular partition and then build a reduced multigraph where two vertices are connected by a double-edge if the corresponding clusters form a regular pair of density larger than $\frac{1}{2}$.
Our proof uses a crucial notion of $K_{r}$-embeddable structures (see Definition~\ref{def2.7}) from the work of Knierim and Su \cite{MR4193066}, which was used to model all possible ways how $K_r$ is embedded into a collection of clusters.
To be more precise, a $K_{r}$-embeddable structure, say $\mathcal{K}$, is a clique in the reduced multigraph such that for $a,b\in \mathbb{N}$ with $a+b=|V(\mathcal{K})|$ and $a+2b=r$, $\mathcal{K}$ has $b$ vertices that are pairwise connected by multiple edges (if $b>1$).
Another key ingredient in their approach is that they build an almost perfect fractional tiling with $K_{r}$-embeddable structures via a novel idea (see Definition~\ref{def2.8} and Lemma~\ref{lem4.4}).
We adapt this approach to $H$-tilings which roughly boils down to embedding vertex-disjoint copies of $H$ in different $K_{r}$-embeddable structures where $r=f(H)$.
To apply Lemma~\ref{lem4.4} in our proof, it suffices to show that for every $K_{r}$-embeddable structure $\mathcal{K}$ with $a,b$ given as above and $V(\mathcal{K})=\{V_1,V_2,\ldots,V_{a+b}\}$, we can find \emph{an almost perfect $H$-tiling in an arbitrary collection of subclusters $V'_i\subseteq V_i$, where $2|V_i'|=|V_j'|$ for every $i\in [a],j\in[a+1,a+b]$.}

Unlike in \cite{MR4193066} where they embed a copy of $K_{r}$ each containing exactly one vertex in $V_i$ and one edge in $V_{j}$ for every $i\in [a],j\in[a+1,a+b]$, we instead need to embed an independent set in each $V_i$ and a forest in $V_j$ for every $i\in [a]$ and $j\in [a+1, a+b]$.
For this we shall use a result of Erd\H{o}s, Hajnal, S\'{o}s and Szemer\'{e}di \cite{{MR716422}} which allows us to embed two forests in a regular pair with density above $\frac{1}{2}$, one in each side, such that they are complete to each other.
Note that the $K_{r}$-embeddable structures may have different orders ranging from $\lceil\frac{r}{2}\rceil$ to $r$. For every $K_{r}$-embeddable structure $\mathcal{K}$ with integers $a,b$ as above, we define a suitably-chosen auxiliary graph $Q(a,b)$ which admits an acyclic partition $\{T_1,T_2,\ldots, T_{a+b}\}$ such that $T_i$ induces an independent set and $2|T_i|=|T_j|$\footnote{This is the origin of the family $\widetilde{H}$.} for every $i\in [a],j\in[a+1,a+b]$. In particular, $Q(a,b)$ has an $H$-factor, and thus it suffices to find an almost perfect $Q(a,b)$-tilings in the same context (see Lemma~\ref{lem2.16} and Corollary~\ref{coro2.17}).

\subsection{Basic notation and organization}
We first introduce some notation throughout the paper. For a graph $G:= G(V, E)$, we write $e(G)=|E(G)|$. We say an independent set $I$ is an $\ell$-\emph{independent set} if $|I|=\ell$. Similarly, we can define an $\ell$-\emph{tree} and an $\ell$-\emph{forest}. The \emph{girth} $g(G)$ of a graph $G$ is the length of a shortest cycle. For $U\subseteq V$, let $G[U]$ be the induced subgraph of $G$ on $U$. Let $G-U:=G[V\backslash U]$. For two subsets $A, B\subseteq V(G)$, we use $E(A, B)$ to denote the set of edges joining $A$ and $B$. We write $N_{G}(v)$ for the set of neighbors of $v$ in $G$. For convenience, we use $d_{G}(v)$ to denote the number of edges which contain $v$ in $G$. We omit the subscript $G$ if the graph is clear from the context. For $m\in \mathbb{N}$, we write $[V]^{\geq m}$ for the family of all subsets $U\subseteq V$ with $|U|\geq m$. For a vertex set $W$ and a positive integer $\ell$, we write $\tbinom{W}{\ell}$ to denote the set of all $\ell$-\emph{subsets} of distinct elements of $W$. We write $V(G)=V_{1}\cup V_{2}$ for a \emph{bipartition} of $G$ if $G[V_{i}]$ is an independent set for each $i\in [2]$. For any integers $a\leq b$, let $[a, b]:=\{i\in \mathbb{Z}: a\leq i\leq b\}$ and $[a]:= [1, a]$.

When we write $\beta\ll \gamma$, we always mean that $\beta, \gamma$ are constants in $(0, 1)$, and $\beta\ll \gamma$ means that there exists $\beta_{0}=\beta_{0}(\gamma)$ such that the subsequent arguments hold for all $0<\beta\leq \beta_{0}$. Hierarchies of other lengths are defined analogously.

The rest of the paper is organized as follows. In Section \ref{sec555}, we will prove Proposition \ref{prop5.111} and Proposition \ref{prop1.8}. In Section \ref{sec2}, we will give a proof of our main result and introduce the regularity lemma. In Section \ref{sec3}, our main work is to find a fractional tiling and establish an embedding lemma. We will prove Lemma \ref{lem3.1} in Section \ref{sec4}. In Section \ref{sec5}, we present some necessary results and tools to introduce the latticed-based absorbing method and prove Lemma \ref{lem3.3}.

\section{Proofs of Proposition \ref{prop5.111} and Proposition \ref{prop1.8}}\label{sec555}
We first give a result which is commonly used in the proof of Proposition \ref{prop5.111} and Proposition \ref{prop1.8}.
%
\begin{lemma}[\cite{ER1959}]\label{lem6.1}
For any $\alpha>0$, $k\in \mathbb{N}$ with $k\geq 3$, there exists an $n$-vertex graph $G$ for sufficiently large $n$ such that $\alpha(G)<\alpha n$ and $g(G)>k$.
\end{lemma}

Now we give the proof of Proposition \ref{prop5.111}.
\begin{proof} [Proof of Proposition \ref{prop5.111}]
Given $\alpha>0$ and $h\in \mathbb{N}$, we shall choose $\frac{1}{n}\ll \alpha$.
Let $\ell:=ar(H)$. We divide the proof into following three cases: $\ell=1$, $\ell=2$ and $\ell\geq 3$.

We first give a construction for every graph $H$ with $\mathrm{hcf}_{2}(H, \text{AP})\geq 2$, which will commonly be used when $\ell=1, 2$. Note that there exists $p\in \{\lfloor\frac{n}{2}\rfloor, \lfloor\frac{n}{2}\rfloor+1\}$ that is non-divisible by $\mathrm{hcf}_{2}(H, \text{AP})$.
Let $G_{0}$ be an $n$-vertex graph with $V(G_{0})=V_{1}\cup V_{2}$ such that $|V_{1}|=p$ and $|V_{2}|=n-p$, and $G_{0}[V_{i}]$ be a complete graph for each $i\in[2]$.
It holds that $\delta(G_{0})\geq \frac{n}{2}-2$ and $\alpha(G_{0})=2$.
Let $H'$ be a copy of $H$ in $G_{0}$.
Since $G_{0}$ is disconnected, $|V(H')\cap V_{1}| \equiv 0 \pmod{\mathrm{hcf}_{2}(H, \text{AP})}$.
Thus, for every $H$-tiling $\mathcal{H}$ in $G_0$, $|\bigcup_{H'\in \mathcal{H}}(V(H')\cap V_{1})| \equiv 0 \pmod{\mathrm{hcf}_{2}(H, \text{AP})}$.
Note that $|V_{1}|=p$ is non-divisible by $\mathrm{hcf}_{2}(H, \text{AP})$.
Hence, $\mathcal{H}$ is not an $H$-factor in $G_0$.

If $\ell=1$, then by the assumption that $\mathrm{hcf}(H, \text{AP})\neq 1$, it holds that $\mathrm{hcf}_{2}(H, \text{AP})\geq 2$. We are done by the construction $G_0$ as above.



If $\ell=2$, then by the assumption that $\mathrm{hcf}(H, \text{AP})\neq 1$, we may assume that $\mathrm{hcf}_{2}(H, \text{AP})=1$ and $\mathrm{hcf}_{1}(H, \text{AP})\geq 3$ as the case $\mathrm{hcf}_{2}(H, \text{AP})\geq 2$ would be handled by $G_0$.
Let $G$ be an $n$-vertex graph with $V(G)=V_{1}\cup V_{2}$, $|V_{1}|=\lfloor\frac{n}{2}\rfloor+1$, $|V_{2}|=\lceil\frac{n}{2}\rceil-1$ and $G[V_{1}, V_{2}]$ be a complete bipartite graph.
Let $G[V_{i}]$ be a subgraph with $\alpha(G[V_{i}])\leq \alpha n$ and $g(G[V_{i}])\geq h+1$ given by Lemma \ref{lem6.1} for each $i\in [2]$.
It holds that $\delta(G)\geq \frac{n}{2}-1$ and $\alpha(G)\leq \alpha n$.
Let $H'$ be a copy of $H$ in $G$.
Then $g(G[V_{i}])\geq h+1$ implies that $V(H')\cap V_{i}$ induces a forest in $H'$ for every $i\in[2]$.
Hence, $\mathcal{P}:=\{V(H')\cap V_{1}, V(H')\cap V_{2}\}$ is an acyclic partition of $H'$ and thus $|(V(H')\cap V_{1})|-|(V(H')\cap V_{2})| \equiv 0 \pmod{\mathrm{hcf}_{1}(H, \text{AP})}$.
For every $H$-tiling $\mathcal{H}$ in $G$, it holds that $|\bigcup_{H'\in \mathcal{H}}(V(H')\cap V_{1})|-|\bigcup_{H'\in \mathcal{H}}(V(H')\cap V_{2})| \equiv 0 \pmod{\mathrm{hcf}_{1}(H, \text{AP})}$.
Note that $\mathrm{hcf}_{1}(H, \text{AP})\geq 3$ and $|V_{1}|-|V_{2}|\in \{1, 2\}$.
Hence, $\mathcal{H}$ is not an $H$-factor in $G$.

If $\ell\geq 3$, then by the assumption that $\mathrm{hcf}(H, \text{AP})\neq 1$, it holds that $\mathrm{hcf}_{1}(H, \text{AP})\neq1$. Let $G$ be an $n$-vertex graph with $V(G)=V_{1}\cup \cdots\cup V_{\ell}$, $|V_{1}|=\lfloor\frac{n}{\ell}\rfloor+1$, $|V_{2}|=\lfloor\frac{n}{\ell}\rfloor$, $\lfloor\frac{n}{\ell}\rfloor-1\leq |V_{i}|\leq \lceil\frac{n}{\ell}\rceil$ for each $i\in [3, \ell]$ and $G[V_{i}, V_{j}]$ be a complete bipartite graph for distinct $i, j\in [\ell]$.
Let $G[V_{i}]$ be a subgraph with $\alpha(G[V_{i}])\leq \alpha n$ and $g(G[V_{i}])\geq h+1$ given by Lemma \ref{lem6.1} for each $i\in [\ell]$.
It holds that $\delta(G)\geq n-(\lfloor\frac{n}{\ell}\rfloor+1)\geq (1-\frac{1}{\ell})n-1$ and $\alpha(G)\leq \alpha n$.
Let $H'$ be a copy of $H$ in $G$.
Then by the same arguments in the case above, it holds that for every $H$-tiling $\mathcal{H}$ in $G$, $|\bigcup_{H'\in \mathcal{H}}(V(H')\cap V_{1})|-|\bigcup_{H'\in \mathcal{H}}(V(H')\cap V_{2})| \equiv 0 \pmod{\mathrm{hcf}_{1}(H, \text{AP})}$.
Note that $\mathrm{hcf}_{1}(H, \text{AP})\neq1$ and $|V_{1}|-|V_{2}|=1$.
Hence, $\mathcal{H}$ is not an $H$-factor in $G$.
\end{proof}

Next we prove Proposition \ref{prop1.8}.

\begin{proof} [Proof of Proposition \ref{prop1.8}]
Given $\alpha>0$ and $h\in \mathbb{N}$, we shall choose $\frac{1}{n}\ll \alpha$.
Let $H$ be an $h$-vertex graph and $\ell:=ar(H)$, $G$ be an $n$-vertex graph with $V(G)=V_{1}\cup \cdots \cup V_{\ell}$, $|V_{1}|=\frac{\sigma(H)}{h}n-1$, $|V_{2}|=\big\lceil\frac{h-\sigma(H)}{(\ell-1)h}n\big\rceil+1$, $\big\lfloor\frac{h-\sigma(H)}{(\ell-1)h}n\big\rfloor\leq |V_{i}|\leq \big\lceil\frac{h-\sigma(H)}{(\ell-1)h}n\big\rceil$ for each $i\in [3, \ell]$ and $G[V_{i}, V_{j}]$ be a complete bipartite graph for distinct $i, j\in [\ell]$.
Let $G[V_{i}]$ be a subgraph with $\alpha(G[V_{i}])\leq \alpha n$ and $g(G[V_{i}])\geq h+1$ given by Lemma \ref{lem6.1} for each $i\in [\ell]$. It holds that $\delta(G)\geq n-\left(\big\lceil\frac{h-\sigma(H)}{(\ell-1)h}n\big\rceil+1\right)\geq\left(1-\frac{1}{ar_{cr}(H)}\right)n-1$ and $\alpha(G)\leq \alpha n$. Let $H'$ be a copy of $H$ in $G$. Then $g(G[V_{i}])\geq h+1$ implies that $H'[V(H')\cap V_{i}]$ is a forest for each $i\in[\ell]$, and $\mathcal{P}:=\{V(H')\cap V_{1}, \dots, V(H')\cap V_{\ell}\}$ is an acyclic partition of $H'$. By the definition of $\sigma(H)$, $|V(H')\cap V_{1}|\geq \sigma(H)$. Since $|V_{1}|=\frac{\sigma(H)}{h}n-1$, there are at most $\big\lfloor\frac{|V_{1}|}{\sigma(H)}\big\rfloor\leq\frac{n}{h}-1$ vertex-disjoint copies of $H$ in $G$. Hence, $G$ has no $H$-factor.
\end{proof}

In the next section, we prove Theorem \ref{thm1.3}.

\section{Proof of Theorem \ref{thm1.3}}\label{sec2}

\subsection{Proof of the main result}\label{sec3}
In this subsection, we introduce the central lemmas that are needed for the proof of our main theorem. This subsection is devoted to explaining how they work together to give the proof of Theorem \ref{thm1.3}. The proofs of these lemmas are then presented in full details in Section \ref{sec3} and Section \ref{sec5} respectively.

A crucial and necessary step in our proof is to find an $H$-tiling in the graph $G$ which covers all but a small set of vertices. The following result guarantees the existence of such an $H$-tiling. The proof of Lemma \ref{lem3.1} will be presented in Section \ref{sec3}.

\begin{lemma}[]\label{lem3.1}
Given $\mu, \delta>0$, an $h$-vertex graph $H$ with $h\in \mathbb{N}$ and $h\geq 3$, there exists $\alpha>0$ such that the following holds for sufficiently large $n$. Let $G$ be an $n$-vertex graph with $\delta(G)\geq \max\left\{\left(1-\frac{2}{f(H)}+\mu\right)n, \left(\frac{1}{2}+\mu\right)n\right\}$ and $\alpha(G)\leq \alpha n$. Then $G$ contains an $H$-tiling which covers all but at most $\delta n$ vertices.
\end{lemma}

Lemma \ref{lem3.3} provides an absorbing set in the graph $G$, whose proof can be found in Section \ref{sec5}.
\begin{lemma}[]\label{lem3.3}
Given $\mu, \gamma$ with $0<\gamma\leq \frac{\mu}{2}$, an $h$-vertex graph $H$ with $h\in \mathbb{N}$ and $h\geq 3$, there exist $\alpha, \xi>0$ such that the following holds for sufficiently large $n$. Let $G$ be an $n$-vertex graph with $\delta(G)\geq \max\left\{\left(1-\frac{2}{f(H)}+\mu\right)n, \left(\frac{1}{2}+\mu\right)n\right\}$ and $\alpha(G)\leq \alpha n$. Then $G$ contains a $\xi$-absorbing set $A$ of size at most $\gamma n$.
\end{lemma}

Now we give the proof of Theorem \ref{thm1.3} using Lemma \ref{lem3.1} and Lemma \ref{lem3.3}.

\begin{proof} [Proof of Theorem \ref{thm1.3}] Given $\mu>0$, $h\in\mathbb{N}$ with $h\geq 3$ and an $h$-vertex graph $H$, we shall choose

\begin{center}
$\frac{1}{n}\ll \alpha\ll \delta\ll \xi\ll\gamma\ll \mu$.
\end{center}
Let $G$ be an $n$-vertex graph with $\delta(G)\geq \max\left\{\left(1-\frac{2}{f(H)}+\mu\right)n, \left(\frac{1}{2}+\mu\right)n\right\}$ and $\alpha(G)\leq \alpha n$. By Lemma \ref{lem3.3} with $\gamma\leq \frac{\mu}{2}$, we find a $\xi$-absorbing set $A\subseteq V(G)$ of size at most $\gamma n$ for some $\xi>0$. Let $G_{1}:= G-A$. Then we have
\begin{center}
$\delta(G_{1})\geq \max\left\{\left(1-\frac{2}{f(H)}+\mu\right)n, \left(\frac{1}{2}+\mu\right)n\right\}-\gamma n\geq \max\left\{\left(1-\frac{2}{f(H)}+\frac{\mu}{2}\right)n, \left(\frac{1}{2}+\frac{\mu}{2}\right)n\right\}.$
\end{center}
Therefore by applying Lemma \ref{lem3.1} on $G_{1}$ with $\delta$, we obtain an $H$-tiling $\mathcal{H}$ that covers all but a set $L$ of at most $\delta n$ vertices in $G_{1}$. Since $\delta\ll \xi$, the absorbing property of $A$ implies that $G[A\cup L]$ contains an $H$-factor, which together with $\mathcal{H}$ forms an $H$-factor in $G$.
\end{proof}

\subsection{Regularity}
To find an almost perfect tiling, an important ingredient in our proof is Szemer\'{e}di's Regularity Lemma. In this paper, we make use of a degree form of the regularity lemma \cite{MR1395865}.
We shall first introduce some notation.
Given a graph $G$ and a pair $(V_{1}, V_{2})$ of vertex-disjoint subsets in $V(G)$, the \emph{density} of $(V_{1}, V_{2})$ is defined as

\begin{center}
$d(V_{1}, V_{2})=\frac{e(V_{1}, V_{2})}{|V_{1}||V_{2}|}$.
\end{center}

\begin{defn}[]\label{def2.1}
Given $\varepsilon>0$, a graph $G$ and a pair $(V_{1}, V_{2})$ of vertex-disjoint subsets in $V(G)$, we say that the pair $(V_1, V_2)$ is $\varepsilon $-\emph{regular} if for all $X\subseteq V_{1}$ and $Y\subseteq V_{2}$ satisfying

\[
X \subseteq V_{1}, |X| \ge \varepsilon |V_{1}| ~\text{and}~ Y \subseteq V_{2}, |Y| \ge \varepsilon |V_{2}|,
\]
we have
\[
|d(X,Y) - d(V_1,V_2)|  \le  \varepsilon.
\]
\end{defn}

\begin{lemma}[\cite{MR1395865}, Slicing Lemma]\label{lem2.2}
Assume $(V_{1}, V_{2})$ is $\varepsilon$-regular with density $d$. For some $\alpha\geq \varepsilon$, let $V_{1}'\subseteq V_{1}$ with $|V_{1}'|\geq \alpha|V_{1}|$ and $V_{2}'\subseteq V_{2}$ with $|V_{2}'|\geq \alpha|V_{2}|$. Then $(V_{1}', V_{2}')$ is $\varepsilon'$-regular with $\varepsilon':=\max\{2\varepsilon, \varepsilon/\alpha\}$ and for its density $d'$ we have $|d'-d|<\varepsilon$.
\end{lemma}

\begin{lemma}[\cite{MR1395865}, Degree form of the Regularity Lemma]\label{lem2.3}

For every $\varepsilon  > 0$ there is an $N = N(\varepsilon )$ such that the following holds for any real number $\beta\in [0, 1]$ and $n\in \mathbb{N}$. Let $G$ be a graph with $n$ vertices.
Then there exist an $(\varepsilon,\beta)$-regular partition $V(G)=V_{0}\cup \cdots \cup V_{k} $ and a spanning subgraph $G' \subseteq G$ with the following properties:

         \item $({\rm 1})$ $ \frac{1}{\varepsilon}\leq k \le N $;

         \item $({\rm 2})$ $|V_{i}| \le \varepsilon  n$ for $i\in [0, k]$ and $|V_{1}|=|V_{2}|=\cdots=|V_{k}| =m$ for some $m\in \mathbb{N}$;

         \item $({\rm 3})$ $d_{G'}(v) > d_{G}(v) - (\beta + \varepsilon )n$ for all $v \in V(G)$;

         \item $({\rm 4})$ each $V_{i}$ is an independent set in $G' $ for $i\in [k]$;

         \item $({\rm 5})$ all pairs $(V_{i}, V_{j})$ are $\varepsilon $-regular (in $G'$) with density 0 or at least $\beta$ for distinct $i, j\neq0$.
\end{lemma}

A widely-used auxiliary graph accompanied with the regular partition is the reduced graph. To differentiate between dense and very dense pairs of partitions, we employ the following definitions of reduced multigraph.

\begin{defn}[Reduced graph]\label{def2.4}
Let $k\in \mathbb{N}$, $\beta, \varepsilon>0$, $G$ be a graph with a vertex partition $V(G)=V_0\cup \cdots \cup V_k$ and $G'\subseteq G$ be a subgraph fulfilling the properties of Lemma \ref{lem2.3}. We denote by $R_{\beta, \varepsilon}$ the \emph{reduced graph} for the $(\varepsilon,\beta)$-partition, which is defined as follows. Let $V(R_{\beta, \varepsilon})=\{V_{1}, \ldots, V_{k}\}$ and for two distinct clusters $V_{i}$ and $V_{j}$ we draw a double-edge between $V_{i}$ and $V_{j}$ if $d_{G'}(V_i, V_j)\geq \frac{1}{2}+\beta$, a single-edge if $\beta\leq d_{G'}(V_i, V_j)<\frac{1}{2}+\beta$ and no edge otherwise.
\end{defn}

The following fact presents a minimum degree of the reduced graph provided the minimum degree of $G$, where a double-edge is counted as two edges.
\begin{fac}[]\label{fact2.5}
Let $n, h\in \mathbb{N}$, $\mu>0$, $0<\varepsilon, \beta\leq \frac{\mu}{10}$ with $\beta\in [0, 1]$, $H$ be an $h$-vertex graph and $G$ be an $n$-vertex graph with $\delta(G)\geq\left(1-\frac{2}{f(H)}+\mu\right)n$. Let $V(G)=V_{0}\cup \cdots \cup V_{k}$ be a vertex partition of $V(G)$ satisfying Lemma \ref{lem2.3} $(1)$-$(5)$. We denote the reduced graph as $R_{\beta, \varepsilon}$. Then for every $V_{i}\in V(R_{\beta, \varepsilon})$ we have

\begin{center}
$d_{R_{\beta, \varepsilon}}(V_{i})\geq2\left(1-\frac{2}{f(H)}+\frac{\mu}{2}\right)k$.
\end{center}
\end{fac}

\begin{proof} Note that $|V_{0}|\leq \varepsilon n$ and $|V_{i}|=m$ for each $i\in [k]$. Every edge in $R_{\beta, \varepsilon}$ represents less than $\left(\frac{1}{2}+\beta\right) m^{2}$ edges in $G'-V_{0}$. Thus we have
\begin{align*}
    d_{R_{\beta, \varepsilon}}(V_{i}) & \geq \frac{|V_{i}|\left(\delta(G)-(\beta+\varepsilon)n-\varepsilon n\right)}{\left(\frac{1}{2}+\beta\right) m^{2}} \\
    & \geq \frac{\left(1-\frac{2}{f(H)}+\mu-2\varepsilon-\beta\right)mn}{\left(\frac{1}{2}+\beta\right) m^{2}} \\
    & \geq 2\left(1-\frac{2}{f(H)}+\mu-2\varepsilon-\beta\right)(1-2\beta)k\\
    & > 2\left(1-\frac{2}{f(H)}+\frac{\mu}{2}\right)k,
\end{align*}
since $0<\varepsilon, \beta\leq \frac{\mu}{10}$ and $\left(\frac{1}{2}+\beta\right)^{-1}\geq 2(1-2\beta)$.
\end{proof}

\begin{rmk}\label{remk2.6}
Let $R$ be a multigraph with multiplicity 2. Note that $|N_{R}(i)|\geq \frac{1}{2}d_{R}(i)$ for each $i\in V(R)$. The \emph{double-edge neighborhood} of $i\in V(R)$ is a set of vertices in $V(R)$ which are connected to $i$ through double-edges. Similarly, we define the \emph{single-edge neighborhood}.
\end{rmk}

\section{Almost perfect tilings}\label{sec3}
To obtain an almost perfect $H$-tiling in the graph $G$, we first define a family $\mathcal{Q}$ of suitably-chosen auxiliary graphs $Q(a,b)$ (to be defined later) for $a, b\in \mathbb{N}$ such that $a+2b=f(H)$ and $Q(a,b)$ contains an $H$-factor. This roughly reduces the problem to finding in $G$ a collection of vertex-disjoint copies of members from $\mathcal{Q}$ which altogether cover almost all vertices.
Here our proof adopts a standard application the regularity lemma on $G$ to get a reduced graph $R$. A key step in it is to construct certain structures in $R$ for embedding $Q(a,b)$. In this case, we use an idea from the work of Knierim and Su \cite{MR4193066} to find a fractional tiling with $K_{f(H)}$-embeddable structures (see Definition~\ref{def2.7}); and then develop a tool (see Corollary~\ref{coro2.17}) for embedding $Q(a,b)$ under certain pseudorandomness conditions.

\subsection{Fractional tilings}
The main result in this subsection is Lemma \ref{lem4.4} which provides us a fractional tiling with some special structures in the reduced graph.
Here, we first present some related notation about these special structures. They are formalised as follows.

\begin{defn}[\cite{MR4193066}, Definition 2.6]\label{def2.7}
Let $R$ be a multigraph with multiplicity 2. Then a $K_{r}$-\emph{multi-embedding} to $R$ is a mapping $\phi: V(K_{r})\rightarrow V(R)$ with the following properties:
\begin{itemize}
  \item for any $i\in V(R)$ the induced subgraph on the vertex set $\phi^{-1}(i)$ (if not empty) in $K_{r}$ is either an isolated vertex or an edge (in particular, $|\phi^{-1}(i)|\le 2$);
  \item if $uv\in E(K_{r})$, then $\phi(u)$ and $\phi(v)$ are connected by at least one edge in $R$ (as long as $\phi(u)$ and $\phi(v)$ differ);
  \item if $|\phi^{-1}(i)|=|\phi^{-1}(j)|=2$ for distinct $i, j\in V(R)$, then $i$ and $j$ are connected by a double-edge.
\end{itemize}
\end{defn}

We also need some definitions which are related to the $K_{r}$-multi-embedding. If $\phi$ is a $K_{r}$-multi-embedding to $R$, then the corresponding subgraph $R[\phi(K_{r})]=:\mathcal{K}$ is a \emph{$K_{r}$-embeddable structure} in $R$. We write $i_{\mathcal{K}}(v)=|\phi^{-1}(v)|$ for every $v\in V(R)$ and by Definition~\ref{def2.7} we know that $i_{\mathcal{K}}(v)\in\{0,1,2\}$. 
We use $\mathcal{F}(R, r):=\{\mathcal{K}_{1},\dots, \mathcal{K}_{\ell}\}$ to denote the family of all $K_{r}$-embeddable structures in $R$.

\begin{defn}[]\label{def2.8}
Let $R$ be a $k$-vertex multigraph with multiplicity 2 and $\mathcal{F}(R, r)$ be given as above. Then a fractional $\mathcal{F}(R, r)$-tiling $\omega$ in $R$ is a weight function from the members of $\mathcal{F}(R, r)$ to the interval $[0, 1]$ such that for every vertex $v\in V(R)$ it holds that
\[
\omega(v) :=\sum\limits_{\mathcal{K}\in\mathcal{F}(R, r)}\omega(\mathcal{K})i_{\mathcal{K}}(v)\leq 1.
\]
\end{defn}
We call $\omega(R):=\sum_{v\in V(G)}\omega(v)$ the \emph{total weight} of the fractional $\mathcal{F}(R, r)$-tiling $\omega$ and it is a \emph{perfect} fractional tiling for $R$ if $\omega(R)=|V(R)|$.
We shall use the following result to find a fractional $\mathcal{F}(R, r)$-tiling with a large total weight, whose proof will be given in Section \ref{sec3.3}.
\begin{lemma}[]\label{lem4.4}
For $h\in \mathbb{N}$, $h\geq3$, an $h$-vertex graph $H$ and positive constants $\mu, \eta$, there exist $\beta, \varepsilon, \alpha>0$ such that the following holds for sufficiently large $n$. Let $G$ be an $n$-vertex graph with $\delta(G)\geq \max\left\{\left(1-\frac{2}{f(H)}+\mu\right)n, \left(\frac{1}{2}+\mu\right)n\right\}$, $\alpha(G)\leq \alpha n$ and $R:=R_{\beta, \varepsilon}$ be a reduced graph for an $(\varepsilon, \beta)$-regular partition of $G$. Then $R$ contains a fractional $\mathcal{F}(R, f(H))$-tiling $\omega$ such that $\omega(R)\geq (1-\eta)|V(R)|$.
\end{lemma}

\subsection{An embedding Lemma}\label{sec2.2}
The main goal of this subsection is to prove an embedding lemma for our purpose. Recall that Lemma~\ref{lem4.4} gives us a large fractional tiling with $K_{f(H)}$-embeddable structures.
Let $\{\mathcal{K}_{1}, \dots, \mathcal{K}_{\ell}\}$ be all the $K_{f(H)}$-embeddable structures in $R$. Roughly speaking, Lemma \ref{lem4.4} tells us that one can correspond the proportion of weight occupied by $\mathcal{K}_{i}$, e.g. $\omega(\mathcal{K}_{i})$ into disjoint vertex sets $W_{i}$ in $G$, in each of which we shall find an almost perfect $H$-tiling. To achieve this, e.g. for $\mathcal{K}_{1}$, we have unique non-negative integers $a, b$ such that $a+b=|V(\mathcal{K}_{1})|$ and $a+2b=f(H)$, and we build an intermediate auxiliary graph $Q(a, b)$ such that we can find an almost perfect $Q(a, b)$-tiling in $W_{1}$ and $Q(a, b)$ itself has an $H$-factor.



To elaborate on this, we first define a graph $Q(a, b, s, F_{1},\ldots, F_{b})$ with given forests $F_{1},\ldots, F_{b}$. 

\begin{defn}[]\label{def2.15}
Let $a, b, s\in \mathbb{N}$ and $F_{i}$ be any $2s$-forest for each $i\in [b]$. Then we construct a graph $Q:=Q(a, b, s, F_{1}, \dots, F_{b})$ with $V(Q)=U_{1}\cup\cdots\cup U_{a+b}$ which satisfies following conditions:
\begin{itemize}
  \item $Q[U_{i}, U_{j}]$ is a complete bipartite graph for distinct $i, j\in [a+b]$;
  \item $Q[U_{i}]$ is an $s$-independent set for $i\in [a]$;
  \item $Q[U_{a+j}]$ is a $2s$-forest $F_{j}$ for $j\in [b]$.
\end{itemize}
\end{defn}

We omit the index $(a, b, s, F_{1}, \dots, F_{b})$ if it is clear from context. By the definition of $Q$, we have $2|U_{i}|=|U_{j}|$ for each $i\in [a]$ and $j\in [a+1, a+b]$.

The following lemma is an essential gadget which allows us to embed an auxiliary graph in $G[V_{1}\cup \cdots\cup V_{a+b}]$ with $a, b$ given as above.


\begin{lemma}[Embedding lemma]\label{lem2.17}
Let $a, b, s$ be positive integers and $\beta>0$. Then there exist $N_0\in \mathbb{N}$ and positive constants $\alpha, \varepsilon$ such that the following holds for any $N\ge N_0$. Let $G$ be a graph with $V(G)=V_{1}\cup \cdots \cup V_{a+b}$, $\alpha(G)\leq \alpha |V(G)|$, $|V_{i}|\geq N$ for each $i\in [a+b]$ such that $(V_{i}, V_{j})$ is $\varepsilon$-regular, $d(V_{i}, V_{j})\geq \beta$ for distinct $i\in [a]$, $j\in [a+b]$ and $d(V_{i}, V_{j})\geq \frac{1}{2}+\beta$ for distinct $i, j\in [a+1, a+b]$. Then for any given $2s$-forests $\{F_{1}, \dots, F_{b}\}$ there exists a copy of $Q(a, b, s, F_{1}, \dots, F_{b})$ in $G$ whose vertex set, say $U_{1}\cup \cdots\cup U_{a+b}$, satisfies $U_{i}\subseteq V_{i}$ for each $i\in [a+b]$.
\end{lemma}

To make use of Lemma \ref{lem2.17}, we shall need the following lemma which guarantees the existence of $Q(a,b)$ as aforementioned.



\begin{lemma}\label{lem2.16}
Let $h\in\mathbb{N}$, $H$ be an $h$-vertex graph and $a, b\in \mathbb{N}$ with $a+2b=f(H)$. Then there exist $s\in \mathbb{N}$ with $s\le h$ and a family of forests $\{F_{1}, \dots, F_{b}\}$ such that $Q(a, b, s, F_{1}, \dots, F_{b})=:Q(a,b)$ contains an $H$-factor.
\end{lemma}

Note that the $Q(a, b)$ is not necessary unique. In the rest of the proof, we fix an instance of $Q(a, b)$ as returned by Lemma \ref{lem2.16}, which serves as building blocks in our proof of Lemma \ref{lem3.1}. For convenience, we formulate this in the following corollary.


\begin{coro}\label{coro2.17}
For any constant $\beta>0$, positive integers $a,b, h$ and an $h$-vertex graph $H$ with $a+2b=f(H)$, there exist $\alpha, \varepsilon>0$ and an integer $s$ with $s\leq h$ such that the following holds for sufficiently large $N$. Let $G$ be a graph with $\alpha(G)\leq \alpha |V(G)|$, $V(G)=V_{1}\cup\cdots \cup V_{a+b}$, $|V_{i}|\ge N$ for each $i\in [a+b]$, $(V_{i}, V_{j})$ be $\varepsilon$-regular with $d(V_{i}, V_{j})\geq \beta$ for distinct $i\in [a]$, $j\in [a+b]$ and $d(V_{i}, V_{j})\geq \frac{1}{2}+\beta$ for distinct $i, j\in [a+1, a+b]$. Then there exists a copy of $Q(a,b)$ in $G$ whose vertex set, say $U_{1}\cup \cdots \cup U_{a+b}$, satisfies $2|U_{i}|=|U_{j}|=2s$ for $i\in [a]$ and $j\in [a+1, a+b]$ and $U_{i}\subseteq V_{i}$ for every $i\in [a+b]$.
\end{coro}

\begin{proof} Given $a, b, h\in \mathbb{N}$, we choose $\frac{1}{N}\ll \alpha\ll\varepsilon\ll \beta$. 
Applying Lemma \ref{lem2.17} with $F_1,F_2,\ldots,F_b$ obtained from Lemma~\ref{lem2.16}, we can embed a copy of $Q(a,b)$ into $G$ as desired.
\end{proof}

\subsection{Proof of Lemma \ref{lem3.1}}\label{sec4}
Equipped with a fractional tiling (Lemma \ref{lem4.4}) in the reduced graph and an embedding lemma (Corollary \ref{coro2.17}), we are able to find an almost perfect $H$-tiling in the original graph $G$.


\begin{proof} [Proof of Lemma \ref{lem3.1}] Given $h\in \mathbb{N}$, an $h$-vertex graph and positive constants $\delta, \mu$, we shall choose
\begin{center}
$\frac{1}{n}\ll \alpha\ll \frac{1}{k}\ll \varepsilon\ll \beta, \eta\ll \delta, \mu, \frac{1}{h}$.
\end{center}
Let $G$ be an $n$-vertex graph with $\delta(G)\geq \max\left\{\left(1-\frac{2}{f(H)}+\mu\right)n, \left(\frac{1}{2}+\mu\right)n\right\}$ and $\alpha(G)\leq \alpha n$.
Applying Lemma \ref{lem2.3} with $\varepsilon, \beta>0$, we obtain an $\varepsilon$-regular partition $\mathcal{P}=\{V_{0}, V_{1}, \dots, V_{k}\}$ of $V(G)$.
Let $m:=|V_{i}|$ for each $i\in [k]$ and $R:=R_{\beta, \varepsilon}$ be a reduced multigraph of the partition $\mathcal{P}$ with multiplicity $2$ and $V(R)=\{V_{1}, \dots, V_{k}\}$.
Write $\mathcal{F}:=\mathcal{F}(R, f(H))=\{\mathcal{K}_{1},\dots, \mathcal{K}_{\ell}\}$.
By applying Lemma \ref{lem4.4} on $R$ with $\eta$, we obtain a fractional $\mathcal{F}$-tiling $\omega$ such that
\begin{center}
$\omega(R)\geq (1-\eta)k$.
\end{center}


Then we construct $\omega'$ from the fractional tiling $\omega$ by scaling the weight function $\omega$ of every $K_{f(H)}$-embeddable-structure with a factor of $(1-\eta)$ i.e. for every $K_{f(H)}$-embeddable-structure $\mathcal{K}$ we have $\omega'(\mathcal{K})=(1-\eta)\omega(\mathcal{K})$. Thus $\omega'$ has total weight at least $(1-2\eta)k$.

For each $K_{f(H)}$-embeddable-structure $\mathcal{K}$ with $\omega'(\mathcal{K})>0$, we have a unique pair of integers $a, b\in \mathbb{N}$ such that $a+b=|V(\mathcal{K})|$ and $a+2b=f(H)$. Write $C_{\mathcal{K}}=\omega'(\mathcal{K})m$. Now we construct a $Q(a,b)$-tiling $\mathcal{Q}_{\mathcal{K}}$ by greedily picking vertex-disjoint copies of $Q(a,b)$ in $G$ such that $\mathcal{Q}_{\mathcal{K}}$ is maximal subject to the fact that it contains at most $i_{\mathcal{K}}(V_i)C_{\mathcal{K}}$ vertices from each $V_i, i\in [k]$. We repeat the process for every $K_{f(H)}$-embeddable-structure $\mathcal{K}$ with positive weight, such that the corresponding $Q(a,b)$-tilings $\mathcal{Q}_{\mathcal{K}}$ are pariwise vertex-disjoint. Note that at the end of this process the set of uncovered vertices in each $V_i$, denoted by $V_i'$, has size \[|V_i'|\ge|V_i|-\sum_{\mathcal{K}\in\mathcal{F}}i_{\mathcal{K}}(V_i)C_{\mathcal{K}}=m-\omega'(V_i)m\ge m-(1-2\eta)m\ge 2\eta m.\]
Now, we claim that every $\mathcal{Q}_{\mathcal{K}}$ covers at least $i_{\mathcal{K}}(V_i)(C_{\mathcal{K}}-h)$ vertices from each $V_i, i\in[k]$. Otherwise, by assuming that $V(\mathcal{K})=\{V_1,V_2,\ldots,V_{a+b}\}$ and applying Lemma~\ref{lem2.2} and Corollary~\ref{coro2.17}, we can pick one more copy of $Q(a,b)$ in $G[V_1'\cup\cdots\cup V_{a+b}']$ which contains at most $i_{\mathcal{K}}(V_i)h$ vertices in each $V_i'$ for $i\in[a+b]$. This contradicts the maximality of $\mathcal{Q}_{\mathcal{K}}$.

Therefore, the total number of vertices covered as above is at least \[\sum_{i\in[k]}\sum_{\mathcal{K}\in\mathcal{F}}(C_{\mathcal{K}}-h)i_{\mathcal{K}}(V_i)=m\sum_{i\in[k]}\sum_{\mathcal{K}\in\mathcal{F}}\omega'(\mathcal{K})i_{\mathcal{K}}(V_i)
-h\sum_{i\in[k]}\sum_{\mathcal{K}\in\mathcal{F}}i_{\mathcal{K}}(V_i)\ge (1-3\eta)mk,\]
where the last inequality follows as $\omega'(R)=\sum_{i\in[k]}\sum_{\mathcal{K}\in\mathcal{F}}\omega'(\mathcal{K})i_{\mathcal{K}}(V_i)\ge(1-2\eta)k$ and $\frac{1}{n}\ll \frac{1}{k}\ll \eps\ll \eta$. As each $Q(a,b)$ contains an $H$-factor and $\eta\ll \delta$, the union of these $\mathcal{Q}_{\mathcal{K}}$ provides an $H$-tiling which covers all but at most $n-(1-3\eta)mk\le \delta n$ vertices in $G$ and this completes the proof.
\end{proof}
In the next subsection, we prove Lemma \ref{lem4.4}, Lemma \ref{lem2.17} and Lemma \ref{lem2.16}.

\subsection{Proof of related lemmas}\label{sec3.3}

\subsubsection{Proof of Lemma \ref{lem4.4}}
The proof of Lemma \ref{lem4.4} relies on the following two results Lemma \ref{lem4.1} \cite{MR4193066} and Lemma \ref{lem4.2} \cite{MR4193066}.

\begin{lemma}[\cite{MR4193066}, Lemma 4.4]\label{lem4.1}
For every $r\in \mathbb{N}$ with $r\ge 4$ and $\eta, \mu>0$, there exist $\alpha>0$ and $n_{0}\in \mathbb{N}$ such that every graph $G$ on $n\geq n_{0}$ vertices with $\delta(G)\geq \left(1-\frac{2}{r}+\mu\right)n$ and $\alpha(G)< \alpha n$ has a fractional $K_{r}$-tiling $\omega$ such that
\begin{center}
$|\{v\in G: \omega(v)<1-\eta\}|\leq \eta n$.
\end{center}
\end{lemma}

\begin{lemma}[\cite{MR4193066}, Lemma 4.7]\label{lem4.2}
For every $r\in \mathbb{N}$ with $r\geq 4$ and $\mu, \eta>0$, there exist $\beta, \varepsilon, \gamma>0$ such that the following holds for sufficiently large $n$. Let $G$ be an $n$-vertex graph with $\delta(G)\geq \left(1-\frac{2}{r}+\mu\right)n$, $\alpha(G)\leq \gamma n$ and $R:=R_{\beta, \varepsilon}$ be a reduced multigraph with multiplicity $2$, $k:=|V(R)|$. There is a graph $\Gamma$ with $\delta(\Gamma)\geq \left(1-\frac{2}{r}+\frac{\mu}{4}\right)|V(\Gamma)|$ and $\alpha(\Gamma)\leq \gamma|V(\Gamma)|$ such that the following holds.

If $\Gamma$ has a fractional $K_{r}$-tiling with total wight at least $(1-\eta)|V(\Gamma)|$, then $G$ contains a $K_{r}$-tiling covering at least $(1-2\eta)n$ vertices. Moreover, $R$ contains a fractional $\mathcal{F}(R,r)$-tiling $\omega$ such that $\omega(R)\geq (1-\eta)k$.
\end{lemma}

The ``moreover'' part of the statement is not a part of the original statement of the Lemma \ref{lem4.2} in \cite{MR4193066} but is stated explicitly in the proof.

For convenience, we need the following corollary.

\begin{coro}\label{coro3.3}
For every $r\in \mathbb{N}$ with $r\geq 4$ and $\mu, \eta>0$, there exist $\beta, \varepsilon, \gamma>0$ such that the following holds for sufficiently large $n$. Let $G$ be an $n$-vertex graph with $\delta(G)\geq \left(1-\frac{2}{r}+\mu\right)n$, $\alpha(G)\leq \gamma n$ and $R:=R_{\beta, \varepsilon}$ be a reduced multigraph with multiplicity $2$. Then $R$ contains a fractional $\mathcal{F}(R,r)$-tiling $\omega$ such that $\omega(R)\geq (1-\eta)|V(R)|$.
\end{coro}

Corollary \ref{coro3.3} comes directly from Lemma \ref{lem4.1} and Lemma \ref{lem4.2} by applying Lemma \ref{lem4.1} on the graph $\Gamma$ in Lemma \ref{lem4.2} where $\frac{\mu}{4}$ plays the role of $\mu$.

Next, we prove Lemma \ref{lem4.4}.


\begin{proof} [Proof of Lemma \ref{lem4.4}] We write $k:=|V(R)|$. If $f(H)\geq 4$, then applying Corollary \ref{coro3.3} with $r=f(H)$ we obtain that there exists a fractional $\mathcal{F}\left(R,f(H)\right)$-tiling $\omega$ and
\begin{center}
$\omega(R)=\sum_{V_{i}\in V(R)}\omega(V_{i})\geq (1-\eta)k$.
\end{center}

If $f(H)=2, 3$, then $\delta(G)\geq \left(\frac{1}{2}+\mu\right)n$. Applying Corollary \ref{coro3.3} with $r=4$, we obtain that there exists a fractional $\mathcal{F}(R,4)$-tiling $\omega$. Next, we construct from $\omega$ a fractional $\mathcal{F}(R,3)$-tiling $\omega_{1}$ and a fractional $\mathcal{F}(R,2)$-tiling $\omega_{2}$ such that $\omega_{1}(V_i)=\omega_{2}(V_i)=\omega(V_i)$ for every vertex $V_i$ in $R$. Note that $\mathcal{F}(R, 4)$ is the family of all $K_{4}$-embeddable structures in $R$. For every $\mathcal{K}\in \mathcal{F}(R, 4)$, if $\mathcal{K}$ is a copy of $K_{4}$ in $R$ and the triangles in the $\mathcal{K}$ are denoted as $\{\mathcal{K}^{1}, \mathcal{K}^{2}, \mathcal{K}^{3}, \mathcal{K}^{4}\}$, then we define $\omega_{1}(\mathcal{K}^{i})=\frac{1}{3}\omega(\mathcal{K})$ for each $i\in [4]$. If $\mathcal{K}$ is a triangle in $R$, say $V_1V_2V_3$ such that $2i_{\mathcal{K}}(V_i)=i_{\mathcal{K}}(V_3)=2, i\in[2]$, then we have three $K_3$-embeddable structures $\mathcal{K}^1,\mathcal{K}^2,\mathcal{K}^3$ defined as follows: $\mathcal{K}^1=V_1V_3$ with $i_{\mathcal{K}^1}(V_1)=1$ and $i_{\mathcal{K}^1}(V_3)=2$; $\mathcal{K}^2=V_2V_3$ with $i_{\mathcal{K}^2}(V_2)=1$ and $i_{\mathcal{K}^2}(V_3)=2$; $\mathcal{K}^3=V_1V_2V_3$ with $i_{\mathcal{K}^3}(V_i)=1$ for every $i\in[3]$. In this case, we define $\omega_{1}(\mathcal{K}^{3})=2\omega_{1}(\mathcal{K}^{i})=\frac{2}{3}\omega(\mathcal{K})$ for each $i\in[2]$. If $\mathcal{K}$ is a double-edge in $R$, say $V_1V_2\in E(R)$ and the multiple edges in the double-edge are denoted as $\{\mathcal{K}^{1}, \mathcal{K}^{2}\}$, then $\mathcal{K}^{1}$ (or $\mathcal{K}^{2}$) can be simply regarded as a $K_{3}$-embeddable structure with $i_{\mathcal{K}^1}(V_1)=1$ and $i_{\mathcal{K}^1}(V_2)=2$ (resp. $i_{\mathcal{K}^2}(V_1)=2$ and $i_{\mathcal{K}^2}(V_2)=1$). Here we define  $\omega_{1}(\mathcal{K}^{i})=\frac{2}{3}\omega(\mathcal{K})$ for each $i\in [2]$.
In all cases, it is easy to see that $\omega_{1}(V_{i})=\omega(V_{i})$ for every $i\in [k]$ and thus $\omega_{1}(R)=\sum_{V_{i}\in V(R)}\omega_{1}(V_{i})=\omega(R)$.

Next we shall construct a fractional $\mathcal{F}(R,2)$-tiling $\omega_{2}$ from $\omega$ such that $\sum_{V_{i}\in V(R)}\omega_{2}(V_{i})\geq (1-\eta)k$. By Definition \ref{def2.7}, every vertex $V_i$ in $V(R)$ is a $K_{2}$-embeddable structure, and we define $\omega_{2}(V_{i})=\omega(V_{i})$ for every $i\in [k]$. Then $\omega_{2}(R)=\sum_{V_{i}\in V(R)}\omega_{2}(V_{i})=\omega(R)$.
\end{proof}
\subsubsection{Proof of Lemma \ref{lem2.17}}

Before the proof of Lemma \ref{lem2.17}, we need several results as follows. The first one is due to Gy\'{a}rf\'{a}s, Szemer\'{e}di and Tuza \cite{MR1395865} and independently, Sumner \cite{MR634555}.

\begin{lemma}[\cite{MR1395865}, \cite{MR634555}]\label{lem2.9}
A k-chromatic graph contains every tree on $k$ vertices as a subgraph.
\end{lemma}
In our context, we shall use the following corollary of Lemma \ref{lem2.9}.

\begin{coro}\label{coro2.10}
Let $n\geq k$ be any integers and $G$ be an $n$-vertex graph with $\alpha(G)\leq \frac{n}{k}$. Then every $k$-tree is contained in $G$.
\end{coro}

The next gadget in our proof is Lemma \ref{lem2.14} proved by Erd\H{o}s, Hajnal, S\'{o}s and Szemer\'{e}di~\cite{MR716422}, which enables us to embed one tree inside the common neighborhood of other trees under
certain density conditions.
Before the statement of Lemma \ref{lem2.14}, we first need the following notation of $(r_{1}, r_{2})$-graphs~\cite{MR716422}.

\begin{defn}[\cite{MR716422}, Definition 2.2]\label{def2.13}
For $r_{1}, r_{2}\in \mathbb{N}$, a graph $G(V, E)$ is said to be an $(r_{1}, r_{2})$-$graph$ with root $v\in V(G)$ if $|V(G)|\leq r_{1}^{r_{2}}+1$ and each $u\in V(G)$ with distance at most $r_{1}$ from $v$ has degree at least $r_{2}$.
\end{defn}

Obviously, for $r\geq 1$ an arbitrary tree of $r+1$ vertices is a subgraph of any $(r, r)$-graph.

\begin{lemma}[\cite{MR716422}, Lemma 2.4]\label{lem2.14}
Given $r_{1}, r_{2}, p\in \mathbb{N}$ and $c>0$, there exist positive constants $c'$ and $s$ such that the following holds for sufficiently large $n$. Let $V_{0}, V_{1}, \dots, V_{p}$ be vertex sets each of size $n$ and $G$ be a graph defined on $V_{0}$ with $|E(G)|\geq sn$. Then for all given mappings $f_{i}: E(G)\rightarrow [V_{i}]^{\geq cn}$ with $i\in[p]$, there exists an $(r_{1}, r_{2})$-graph $H_{1}\subseteq G$ with

\begin{center}
$\big|\bigcap_{e\in E(H_{1})} f_{i}(e)\big|\geq c' n$ for every $i\in [p]$.
\end{center}

\end{lemma}

Note that to invoke Lemma \ref{lem2.14} in the proof of Lemma \ref{lem2.17}, we need Proposition \ref{prop2.12} which gives a lower bound on the number of edges in the setting of small independence number.

\begin{prop}[]\label{prop2.12}
Let $G$ be an $n$-vertex graph with $\alpha(G)\leq \alpha n$. Then $e(G)\geq \frac{(1-\alpha)n}{2\alpha}$.
\end{prop}

\begin{proof}
Note that
\begin{align*}
  1+\frac{2e(G)}{n} & =\frac{\sum_{v \in V(G)}\left(d(v)+1\right)}{n} \geq \frac{n}{\sum_{v \in V(G)}\frac{1}{d(v)+1}}\ge\frac{1}{\alpha},
\end{align*}
where the first inequality follows from the fact that the arithmetic mean is at least the harmonic mean, and the last inequality follows from a well-known result in \cite{MR1090733} stating that $\alpha(G) \geq\sum_{v\in V(G)}\frac{1}{d(v)+1}$.
Thus $e(G)\geq \frac{(1-\alpha)n}{2\alpha}$.
\end{proof}

Next we prove Lemma~\ref{lem2.17}.

\begin{proof}[Proof of Lemma \ref{lem2.17}] Given $a, b, s\in \mathbb{N}$ and $\beta>0$, we shall choose
\[\frac{1}{N}\ll \alpha\ll\varepsilon\ll \varepsilon_{a+b-1}\ll \cdots\ll \varepsilon_{1}\ll c'\ll \beta,\frac{1}{s}\] and let $G, F_1,\ldots,F_b$ be given in Lemma~\ref{lem2.17}.

The proof will proceed by induction on $a+b$. The base case $a+b=1$ is clear as either $Q$ is an $s$-independent set or a $2s$-forest $F_{1}$: If $a=1$ and $b=0$, then we only need to choose a vertex set $U_{1}\subseteq V_{1}$ with $|U_{1}|=s$ which is easily derived since $|V_{1}|\geq N \geq s$. If $a=0$ and $b=1$, then we only need to embed a $2s$-forest $F_{1}$ in $G[V_{1}]$. By Corollary \ref{coro2.10} with $\alpha\leq \frac{1}{2s}$, $G[V_{1}]$ contains every $2s$-forest.


Next we show that our statement holds for $a+b=k$ assuming it holds for $a+b=\ell<k$. 
First assume $a\geq 1$. Since $(V_{i}, V_{j})$ is $\varepsilon$-regular in $G$ for distinct $i, j\in [1, a+b]$, there exists a subset $V_{1}'\subseteq V_{1}$ such that every vertex in $V_{1}'$ has at least $\left(d(V_{1}, V_{j})-\varepsilon\right)|V_{j}|$ neighbors in $V_{j}$ for each $j\in [2, a+b]$ and $|V_{1}'|\geq \left(1-\varepsilon(a+b-1)\right)|V_{1}|$. In $V'_{1}$, we choose $s$ vertices $\{v_{1}, \dots, v_{s}\}$ with $S_{j}^{1}:=N(v_{1})\cap V_{j}$, $S_{j}^{i}:= N(v_{i})\cap S_{j}^{i-1}$ and $|N(v_{i})\cap S_{j}^{i-1}|\geq (\beta-\varepsilon)|S_{j}^{i-1}|\geq \frac{\beta}{2}|S_{j}^{i-1}|$ for each $i\in [2, s]$ and each $j\in [2, a+b]$. Then the $s$ vertices in $V'_{1}$ have a common neighborhood $S_{j}:=S_{j}^{s}$ in $V_{j}$ for each $j\in [2, a+b]$ with $|S_{j}|\geq (\frac{\beta}{2})^{s}|V_{j}|$. Applying Lemma \ref{lem2.2} to $G[\bigcup_{i=2}^{a+b} S_{i}]$, we obtain that $(S_{i}, S_{j})$ is $\varepsilon'$-regular with $\varepsilon':=\max\left\{2\varepsilon, \frac{\varepsilon2^{s}}{\beta^{s}}\right\}=\frac{\varepsilon2^{s}}{\beta^{s}}$ and $d(S_{i}, S_{j})\geq d(V_{i}, V_{j})-\varepsilon\geq d(V_{i}, V_{j})-\frac{\beta}{2}$ for distinct $i, j\in [2, a+b]$. Since $\varepsilon\ll \varepsilon_{a+b-1}$, it holds that $(S_{i}, S_{j})$ is $\varepsilon_{a+b-1}$-regular for distinct $i, j\in [2, a+b]$.
Note that $|S_{j}|\geq (\frac{\beta}{2})^{s} |V_{j}|\geq (\frac{\beta}{2})^{s} N$. We apply the induction hypothesis to find in $G[\bigcup_{i=2}^{a+b} S_{i}]$ a copy of $Q(a-1, b, s, F_{1}, \dots, F_{b})$ which together with $\{v_{1}, \dots, v_{s}\}$ forms a copy of $Q(a, b, s, F_{1}, \dots, F_{b})$ as desired.

Now assume $a=0$ and $b\geq 2$. Recall that $d(V_{i}, V_{j})\geq \frac{1}{2}+\beta$ for distinct $i, j\in [b]$. Inside $V_{1}$, there exists a subset $V_{1}'$ such that every vertex in $V_{1}'$ has at least $(d(V_{1}, V_{j})-\varepsilon)|V_{j}|\geq \left(\frac{1}{2}+\frac{\beta}{2}\right)|V_{j}|$ neighbors in $V_{j}$ for every $j\in [2, b]$ and
\[
|V'_{1}|\geq |V_{1}|-\varepsilon |V_{1}|b=(1-\varepsilon b)|V_{1}|\geq (1-\varepsilon b)N.
\]
Our next step is to embed $Q(a, b, s, F_{1}, \dots, F_{b})$ into $G[V'_{1}\cup V_{2} \cup \cdots \cup V_{b}]$.
Note that for $u, v\in V'_{1}$ and $i\in [2, b]$, we have
\begin{equation}\label{eq1}
|N(u)\cap N(v)\cap V_{i}|\geq \beta|V_{i}|.
\end{equation}
By Proposition \ref{prop2.12}, we have
\[e(V'_{1})\geq \frac{(1-\alpha)}{2\alpha}|V'_{1}|\geq \frac{(1-\alpha)}{2\alpha}(1-\varepsilon b)N.\]
To apply Lemma \ref{lem2.14}, we define for every $i\in [2, b]$ a mapping
\begin{center}
$f_{i}: E(G[V'_{1}])\rightarrow [V_{i}]^{\geq \beta |V_{i}|}$,
\end{center}
by letting $f_{i}(uv):= N(u)\cap N(v)\cap V_{i}$ for every $uv\in E(G[V'_{1}])$. From (\ref{eq1}), it holds that $|f_{i}(uv)|\geq \beta |V_{i}|$ for every $uv\in E(G[V'_{1}])$. Lemma \ref{lem2.14} applied with $r_{1}=r_{2}=2s$ and $c=\beta$ implies that there exist a constant $c'$ and a $(2s, 2s)$-subgraph $H_{1}\subseteq G[V'_{1}]$ such that for $i\in[2, b]$, $|\bigcap_{e\in E(H_{1})} f_{i}(e)|\geq c' |V_{i}|$. By definition, $H_{1}$ contains a subgraph isomorphic to $F_{1}$. Write $S_{i}:= \bigcap_{e\in E(F_{1})} f_{i}(e)$ for each $i\in [2, b]$. Then for every $i\in [2, b]$, we have $|S_{i}| \geq c'|V_{i}|\geq c'N$. By Lemma \ref{lem2.2}, we have that $(S_{i}, S_{j})$ is $\varepsilon'$-regular where $\varepsilon':=\max\left\{2\varepsilon, \frac{\varepsilon}{c'}\right\}\leq \varepsilon_{a+b-1}$ since $\varepsilon\ll \varepsilon_{a+b-1}, c'$ and $d(S_{i}, S_{j})\geq \frac{1}{2}+\frac{\beta}{2}$. By induction hypothesis, $G[\bigcup_{i=2}^{b} S_{i}]$ contains a copy of $Q(a, b-1, s, F_{2}, \dots, F_{b})$ which together with $F_{1}$ forms a copy of $Q(a, b, s, F_{1}, \dots, F_{b})$ as desired.
\end{proof}

\subsubsection{Proof of Lemma \ref{lem2.16}}

\begin{proof}[Proof of Lemma \ref{lem2.16}] Let $r:=ar(H)$ and $\mathcal{T}=\{T_{0}, T_{1}, \dots, T_{r-1}\}$ be an acyclic partition of $H$.
For every $T_{i}\in \mathcal{T}$, $\mathcal{P}_{i}=\{V_{i,1}, V_{i,2}\}$ is a bipartition of $T_{i}$, and $T_{i, j}:=V_{i,j}$ is an independent set in $H$ for $j\in[2]$.
We divide the proof into the following two cases.

If $f(H)=2r$, then we take $s=h$ and construct $F_{1}= F_2=\cdots=F_{b}$ such that $F_{1}$ is a $2h$-forest which consists of two vertex-disjoint copies of $H[T_{j}]$ for each $j\in [0, r-1]$. Recall that $\frac{a}{2}+b=r$. To show that $Q(a, b, h, F_{1}, \dots, F_{b})$ contains an $H$-factor, we build an auxiliary matrix \[A=\{a_{ij}\}_{r\times r}=\begin{pmatrix}
	T_{1} & T_{2} & \dots & T_{0}\\
	T_{2} & T_{3} & \dots & T_{1} \\
    \vdots & \vdots& \ddots & \vdots\\
	T_{0} & T_{1} &\dots & T_{r-1}
\end{pmatrix}\]
where $a_{ij}:=T_{i+j-1\pmod{r}}$ for $i, j\in [r]$. Note that every row (or column) of $A$ corresponds to a permutation of $\mathcal{T}=\{T_{0}, T_{1}, \dots, T_{r-1}\}$. Now we construct two matrices $A_{1}$ and $A_{2}$ as follows by ``cutting'' each of the first $\frac{a}{2}$ columns in half and then swapping columns two by two: \[A_{1}=\{a^{1}_{i, j}\}_{r\times (a+b)}=\begin{pmatrix}
	T_{1,1} & T_{1,2}& \dots & T_{\frac{a}{2},1} & T_{\frac{a}{2},2} & T_{\frac{a}{2}+1} & \dots & T_{0}\\
	T_{2,1} & T_{2,2}& \dots & T_{\frac{a}{2}+1,1} & T_{\frac{a}{2}+1,2} & T_{\frac{a}{2}+2} & \dots & T_{1}\\
    \vdots &\vdots & \ddots &\vdots & \vdots & \vdots & \ddots & \vdots\\
	T_{0,1} & T_{0,2}& \dots & T_{\frac{a}{2}-1,1} & T_{\frac{a}{2}-1,2} & T_{\frac{a}{2}} & \dots & T_{r-1}\\
\end{pmatrix}\]
where $a^{1}_{i, 2j-1}:=T_{i+j-1\pmod{r}, 1}$ and $a^{1}_{i, 2j}:=T_{i+j-1\pmod{r}, 2}$ for each $i\in [r]$ and $j\in[\frac{a}{2}]$; \[A_{2}=\{a^{2}_{i, j}\}_{r\times (a+b)}=\begin{pmatrix}
	T_{1,2} & T_{1,1}& \dots &T_{\frac{a}{2},2} & T_{\frac{a}{2},1} & T_{\frac{a}{2}+1} & \dots & T_{0}\\
	T_{2,2} & T_{2,1}& \dots &T_{\frac{a}{2}+1,2} & T_{\frac{a}{2}+1,1} & T_{\frac{a}{2}+2} & \dots & T_{1}\\
    \vdots &\vdots & \ddots &\vdots & \vdots & \vdots & \ddots & \vdots\\
	T_{0,2} & T_{0,1}& \dots &T_{\frac{a}{2}-1,2} & T_{\frac{a}{2}-1,1} & T_{\frac{a}{2}} & \dots & T_{r-1}\\
\end{pmatrix}\]
where $a^{2}_{i, 2j-1}:=T_{i+j-1\pmod{r}, 2}$ and $a^{2}_{i, 2j}:=T_{i+j-1\pmod{r}, 1}$ for each $i\in [r]$ and $j\in[\frac{a}{2}]$.


Then $A^{\ast}: =\begin{pmatrix}
	 A_{1}\\
     A_{2}
\end{pmatrix}$ is a $2r\times (a+b)$ matrix and observe that for each $i\in[a]$, disjoint union of the elements taken over the $i$th column gives us an $h$-independent set, whilst each of the last $b$ columns of $A^{\ast}$ provides a copy of $F_1$. Clearly, $Q(a, b, h, F_{1}, \dots, F_{b})$ has an $H$-factor with $2r$ vertex-disjoint copies of $H$.

If $f(H)=2r-1$, then we take $s=\frac{2h}{2r-1}$. By the definition of $\widetilde{\mathcal{H}}$, we know that in $\mathcal{T}=\{T_{0}, T_{1}, \dots, T_{r-1}\}$, $T_0$ is an independent set in $H$ and $|T_i|=2|T_0|, i\in[r-1]$. Thus $|T_0|=\frac{h}{2r-1}$ and recall that $\frac{a+1}{2}+b=r$. Now we construct $F_{1},F_2,\ldots ,F_{b}$ such that each $F_{i}$ is a $4|T_0|$-forest which consists of two vertex-disjoint copies of $H[T_{r-b+i-1}]$.

To show that $Q(a, b, \frac{2h}{2r-1}, F_{1}, \dots, F_{b})$ contains an $H$-factor, we build an auxiliary matrix $A=\{a_{1j}\}_{1\times r}=\begin{pmatrix}
	T_{0}&\dots& T_{r-1}
\end{pmatrix}$
where $a_{1j}:=T_{j}$ for each $j\in [0, r-1]$. It holds that $A$ corresponds to an acyclic partition of $H$. Similarly, we construct two matrices $A_{1}$ and $A_{2}$ as follows: \[A_{1}=\{a^{1}_{i, j}\}_{1\times (a+b)}=\begin{pmatrix}
	T_{0} & T_{1,1} & T_{1,2} & \dots & T_{\frac{a-1}{2},1} & T_{\frac{a-1}{2},2} & T_{r-b} & \dots & T_{r-1}
\end{pmatrix},\] \[A_{2}=\{a^{2}_{i, j}\}_{1\times (a+b)}=\begin{pmatrix}
	T_{0} & T_{1,2} & T_{1,1} & \dots & T_{\frac{a-1}{2},2} & T_{\frac{a-1}{2},1} & T_{r-b} & \dots & T_{r-1}
\end{pmatrix}.\]



Let $A^{\ast}: =\begin{pmatrix}
	 A_{1}\\
     A_{2}
\end{pmatrix}$ be a $2\times (a+b)$ matrix. By taking a disjoint union of the elements in the columns as above, each of the first $a$ columns of $A^{\ast}$ provides a $\frac{2h}{2r-1}$-independent set while the last $b$ columns respectively provides copies of $F_i, i\in [b]$. Clearly, $Q(a, b, \frac{2h}{2r-1}, F_{1}, \dots, F_{b})$ has an $H$-factor containing two vertex-disjoint copies of $H$.
\end{proof}

\section{Absorbing}\label{sec5}

In this section, we give a proof of Lemma \ref{lem3.3}. We shall use a result of Nenadov and Pehova \cite{MR4080942} which gives a sufficient condition on the existence of an absorbing set.
\begin{lemma}[\cite{MR4080942}, Lemma 2.2]\label{lem5.1}
Let $H$ be a graph with $h$ vertices, $\gamma>0$ and $t\in \mathbb{N}$ be constants. Then there exists $\xi:=\xi(h, t, \gamma)$ such that the following holds for sufficiently large $n$. Suppose that $G$ is a graph with $n$ vertices such that every $S\in \binom{V(G)}{h}$ has a family of at least $\gamma n$ vertex-disjoint $(H, t)$-absorbers. Then $G$ contains a $\xi$-absorbing set of size at most $\gamma n$.
\end{lemma}
So the key point in the proof of Lemma \ref{lem3.3} is to build linearly many vertex-disjoint absorbers for every $S\in \binom{V(G)}{h}$.
To achieve this, we employ the latticed-based absorbing method \cite{HMWY2021} and we first need the notion of $H$-reachability from \cite{HMWY2021} which originates in \cite{MR3338027}.

\begin{defn}
  Let $G, H$ be given as aforementioned and $m, t\in \mathbb{N}$. Then we say that two vertices $u, v\in V(G)$ are \emph{$(H, m, t)$-reachable} (in $G$) if for any vertex set $W$ of $m$ vertices, there is a set $S\subseteq V(G)\backslash W$ of size at most $ht-1$ such that both $G[S\cup \{u\}]$ and $G[S\cup \{v\}]$ have $H$-factors, where we call such $S$ an \emph{$H$-connector} for $u, v$. Moreover, a set $U\subseteq V(G)$ is \emph{$(H, m, t)$-closed} if every two vertices $u, v\in U$ are $(H, m, t)$-reachable, where the corresponding $H$-connector for $u, v$ may not be contained in $U$. If two vertices $u, v\in V(G)$ are $(H, m, 1)$-reachable, then we say $u$ is \emph{1-reachable} to $v$. If $u, v\in U$ are $(H, m, t)$-reachable, and the corresponding $H$-connector for $u, v$ is contained in $U$, then we say that $u, v\in U$ are \emph{$(H, m, t)$-inner-reachable}. Similarly, we can define \emph{$(H, m, t)$-inner-closed} and \emph{1-inner-reachable}.
\end{defn}



The following result from \cite{HMWY2021} builds a sufficient condition to ensure that every subset $S\subseteq V(G)$ with $|S|=h$ has linearly many vertex-disjoint absorbers.

\begin{lemma}[\cite{HMWY2021}, Lemma 3.9]\label{lem5.2}
Given $\beta>0$, $t, h\in \mathbb{N}$ with $h\geq 3$ and an $h$-vertex graph $H$, the following holds for sufficiently large $n\in \mathbb{N}$. Let $G$ be an $n$-vertex graph such that $V(G)$ is $(H, \beta n, t)$-closed. Then every $S\in \tbinom{V(G)}{h}$ has a family of at least $\frac{\beta}{h^{2}t}n$ vertex-disjoint $(H, t)$-absorbers.
\end{lemma}

Based on this lemma, it suffices to show that $V(G)$ is closed. However, we are only able to prove a slightly weaker result which states that the graph $G$ admits a vertex partition $V(G)=B\cup U$ where $B$ is a small vertex set and $U$ is inner-closed.

\begin{lemma}[]\label{lem7.1}
Given $h\in \mathbb{N}$ with $h\geq 3$, an $h$-vertex graph $H$ and constants $\tau, \mu$ with $0<\tau<\mu$, there exist positive constants $\alpha, \beta$ and $t\in \mathbb{N}$ such that the following holds for sufficiently large $n$. Let $G$ be an $n$-vertex graph with $\delta(G)\geq \max\left\{\left(1-\frac{2}{f(H)}+\mu\right)n, \left(\frac{1}{2}+\mu\right)n\right\}$ and $\alpha(G)\leq \alpha n$. Then $G$ admits a partition $V(G)=B\cup U$ with $|B|\leq \tau n$ and $U$ is $(H, \beta n, t)$-inner-closed.
\end{lemma}
Clearly, we shall focus on the subgraph $G[U]$ and obtain an absorbing set by applying Lemma~\ref{lem5.2} and Lemma~\ref{lem5.1} on $G[U]$.

The next step is to deal with the vertex set $B$. We shall pick mutually vertex-disjoint copies of $H$ each covering a vertex in $B$. To achieve this, we use the following result which enables us to find linearly many copies of $H$ covering any given vertex.

\begin{lemma}[]\label{lem5.7}
Given $h\in \mathbb{N}$ with $h\geq 3$, an $h$-vertex graph $H$ and a constant $\mu>0$, there exists $\alpha>0$ such that the following holds for sufficiently large $n$. Let $G$ be an $n$-vertex graph with $\delta(G)\geq \max\left\{\left(1-\frac{2}{f(H)}+\mu\right)n, \left(\frac{1}{2}+\mu\right)n\right\}$ and $\alpha(G)\leq \alpha n$. If $W$ is a subset of $V(G)$ with $|W|\leq \frac{\mu}{2}n$, then there exists at least one copy of $H$ covering $v$ in $G-W$ for each $v\in V(G)\backslash W$.
\end{lemma}

Now we give the proof of Lemma \ref{lem3.3} by using Lemma \ref{lem5.1}, Lemma \ref{lem5.2}, Lemma \ref{lem7.1} and Lemma \ref{lem5.7}.


\begin{proof} [Proof of Lemma \ref{lem3.3}]
Let $h\in \mathbb{N}$, $H$ be an $h$-vertex graph and constants $\gamma, \mu$ with $0<\gamma\leq \frac{\mu}{2}$. Then we shall choose $\tau=\frac{\gamma}{2h}$ and
\begin{center}
$\frac{1}{n}\ll \alpha\ll \xi\ll \frac{1}{t}, \beta\ll \gamma, \mu$.
\end{center}
Let $G$ be an $n$-vertex graph with $\delta(G)\geq \max\left\{\left(1-\frac{2}{f(H)}+\mu\right)n, \left(\frac{1}{2}+\mu\right)n\right\}$ and $\alpha(G)\leq \alpha n$.
Applying Lemma \ref{lem7.1} on $G$, we obtain that $G$ admits a vertex partition $V(G)=B\cup U$ such that $|B|\leq \tau n$ and $U$ is $(H, \beta n, t)$-inner-closed.
Applying Lemma \ref{lem5.2} on $G[U]$, it holds that every $S\in \tbinom{U}{h}$ has a family of at least $\frac{\beta}{h^{2}t}|U|\geq \frac{\beta}{2h^{2}t}n$ vertex-disjoint $(H, t)$-absorbers. Applying Lemma \ref{lem5.1} on $G[U]$ where $\frac{\gamma}{2}$ plays the role of $\gamma$, we obtain a $\xi$-absorbing set $A_{1}$ in $G[U]$ of size at most $\frac{\gamma}{2}n$.

Now we shall iteratively pick vertex-disjoint copies of $H$ each covering one vertex in $B$ while avoiding using any vertex in $A_{1}$, and we claim that every vertex in $B$ can be covered in this way. Let $G_{2}:= G-A_{1}$. For each $v\in B$, we apply Lemma \ref{lem5.7} iteratively to find a copy of $H$ covering $v$ in $G_{2}$, while avoiding $A_{1}$ and all copies of $H$ found so far. This is possible as during the process the number of vertices that we need to avoid is at most

\begin{center}
$h|B|+|A_{1}|\leq h\tau n+ \frac{\gamma}{2}n=\gamma n\leq \frac{\mu}{2}n$.
\end{center}

Let $W$ be the union of the vertex sets over all the $|B|$ vertex-disjoint copies of $H$ as above and $A:= A_{1}\cup W$. Recall that $A_{1}$ is a $\xi$-absorbing set for $G[U]$, and $G[W]$ has an $H$-factor. Thus $A$ is a $\xi$-absorbing set for $G$ with $|A|\leq \gamma n$.
\end{proof}

Now it remains to prove Lemma \ref{lem7.1} and Lemma \ref{lem5.7} whose proofs will be given in next two subsections respectively.

\subsection{Proof of Lemma \ref{lem7.1}}
To prove Lemma \ref{lem7.1}, we divide the proof into two steps (i): $G$ admits a partition $V(G)=B\cup U$ where $B$ is a small vertex set and every vertex in $U$ is $1$-inner reachable to linearly many vertices; (ii): $U$ is inner-closed. 

The following result is the first step.

\begin{lemma}[]\label{lem5.5}
Given $h\in \mathbb{N}$ with $h\geq 3$, an $h$-vertex graph $H$ and constants $\tau, \mu$ with $0<\tau<\mu$, there exist positive constants $\alpha, \beta_{1}, \gamma_{1}$ such that the following holds for sufficiently large $n$. Let $G$ be an $n$-vertex graph with $\delta(G)\geq \max\left\{\left(1-\frac{2}{f(H)}+\mu\right)n, \left(\frac{1}{2}+\mu\right)n\right\}$ and $\alpha(G)\leq \alpha n$. Then $G$ admits a vertex partition $V(G)=B\cup U$ such that $|B|\leq \tau n$ and every vertex in $U$ is $(H, \beta_{1}n, 1)$-inner-reachable to at least $\gamma_{1}n$ other vertices in $G[U]$.
\end{lemma}

In the second step, we only need to apply the following result on $G[U]$.

\begin{lemma}[]\label{lem5.4}
Given $h\in \mathbb{N}$ with $h\geq 3$, an $h$-vertex graph $H$ and constants $\mu, \beta_{1}, \gamma_{1}$ with $0<\mu, \beta_{1}, \gamma_{1}<1$, there exist positive constants $\alpha, \beta$ and $t\in \mathbb{N}$ such that the following holds for sufficiently large $n$. Let $G$ be an $n$-vertex graph with $\delta(G)\geq \max\left\{\left(1-\frac{2}{f(H)}+\mu\right)n, \left(\frac{1}{2}+\mu\right)n\right\}$ and $\alpha(G)\leq \alpha n$ such that every vertex in $V(G)$ is $(H, \beta_{1}n, 1)$-reachable to at least $\gamma_{1}n$ other vertices. Then $V(G)$ is $(H, \beta n, t)$-closed.
\end{lemma}

Obviously, Lemma \ref{lem7.1} is an immediate corollary of the above-mentioned two lemmas. In the following, we will give the proofs of Lemma \ref{lem5.5} and Lemma \ref{lem5.4}.

\subsubsection{Proof of Lemma \ref{lem5.5}: $1$-reachability}
The proof of Lemma \ref{lem5.5} goes roughly as follows. We first apply the regularity lemma on $G$ to obtain a partition and a reduced graph $R$ with multiplicity 2. A result of Knierim and Su \cite{MR4193066} guarantees that every cluster $V_i$ is covered by a $K_{f(H)+1}$-embeddable structure, say $\mathcal{K}_i$ in the reduced graph $R$ (see Lemma~\ref{lem5.11}). In this case, for each $V_i$, by using Corollary~\ref{coro2.17} on $\mathcal{K}_i$, we are able to show that almost all vertices in $V_i$ are $1$-reachable to linearly many other vertices from $V_i$, where the bad vertices would be given iteratively at each stage of the process.  

As sketched above, we need the following result which investigates the structure around every cluster in the reduced graph.

\begin{lemma}[\cite{MR4193066}, Lemma 3.7]\label{lem5.11}
For $r\geq 4$ and $k\in \mathbb{N}$, let $R$ be a multigraph with multiplicity 2 on $k$ vertices and $\delta(R)> \left(1-\frac{2}{r}\right)2k$. Then any double-edge $ij\in E(R)$ is contained in some $K_{r+1}$-embeddable structure.
\end{lemma}

\begin{proof} [Proof of Lemma \ref{lem5.5}]Given $h\in \mathbb{N}$, an $h$-vertex graph $H$ and constants $\tau, \mu$ with $0<\tau< \mu$, we shall choose

\begin{center}
$\frac{1}{n}\ll\frac{1}{N}\ll\alpha\ll \beta_{1}, \gamma_{1}\ll\frac{1}{k}\ll \varepsilon \ll \tau, \mu$.
\end{center}
Let $\beta=\frac{\mu}{10}$, $G$ be an $n$-vertex graph with $\delta(G)\geq \max\left\{\left(1-\frac{2}{f(H)}+\mu\right)n, \left(\frac{1}{2}+\mu\right)n\right\}$ and $\alpha(G)\leq \alpha n$.
Applying Lemma \ref{lem2.3} on $G$ with $\varepsilon, \beta>0$, we obtain an $\varepsilon$-regular partition $\mathcal{P}=\{V_{0}, V_{1}, \dots, V_{k}\}$ of $G$.
Let $m:=|V_{i}|$ for each $i\in [k]$ and $R:=R_{\beta, \varepsilon}$ be a reduced multigraph with multiplicity $2$ of the $\varepsilon$-regular partition $\mathcal{P}$. By Fact \ref{fact2.5}, we have $\delta(R)\geq 2\left(1-\frac{2}{f(H)}+\frac{\mu}{2}\right)k$ when $f(H)\geq 4$ and $\delta(R)\geq (1+\mu)k$ when $f(H)=2, 3$. Clearly, every cluster is in a double-edge. Applying Lemma \ref{lem5.11} with $r=\max\{f(H), 4\}$, we have that every cluster is in a $K_{f(H)+1}$-embeddable structure when $f(H)\geq 4$ and every cluster is in a $K_{5}$-embeddable structure when $f(H)=2, 3$. Therefore every cluster is covered by a $K_{f(H)+1}$-embeddable structure.

Write $\mathcal{F}:=\mathcal{F}(R, f(H)+1)$ for the family of all $K_{f(H)+1}$-embeddable structures in $R$. Since every cluster is in some $K_{f(H)+1}$-embeddable structure, there exists a minimal subfamily $\{\mathcal{K}_{1}, \dots, \mathcal{K}_{\ell}\}$ such that $V(R)=\bigcup_{i=1}^{\ell}$ $V(\mathcal{K}_{i})$. Now we first define the vertex set $B$ as follows:

\begin{itemize}
  \item [(1)]
  Let $S_{i}=V(\mathcal{K}_{i})\backslash \bigcup_{p=1}^{i-1} V(\mathcal{K}_{p})$ for $i\in [\ell]$, where $S_{1}=V(\mathcal{K}_{1})$. Then it follows from the minimality that $S_{i}\neq \emptyset$ and we write $S_{i}=\{V_{i_{1}}, \dots, V_{i_{s_{i}}}\}$ for some integer $s_{i}:=|S_{i}|$;
  \item [(2)]
  For $i\in [\ell]$ and $j\in [s_{i}]$, $B_{i_{j}}:=\Big\{v\in V_{i_{j}}\Big| |N(v)\cap V_{s}|\leq (d(V_{i_{j}}, V_{s})-\varepsilon)|V_{s}|$ for some $V_{s}\in V(\mathcal{K}_{i})\, \text{with}\, s\neq i_{j}\Big\}$, and $B_{i}:=\bigcup_{j=1}^{s_{i}} B_{i_{j}}$;
  \item [(3)]
  $B:=\bigcup_{i=1}^{\ell} B_{i}$.
\end{itemize}


Observe that $\{S_{1}, S_{2}, \dots, S_{\ell}\}$ is a partition of $V(R)$ and $|S_{i}|\le 2r+1$, $i\in [\ell]$. Moreover, for every $i\in [\ell]$ and $j\in [s_{i}]$, we have $|B_{i_{j}}|\le \varepsilon m|\mathcal{K}_{i}|$. Thus $|B|\leq \varepsilon m(2r+1)k\leq \tau n$ since $\varepsilon \ll \tau$.

Let $U:=V(G)\backslash B=\bigcup_{i=1}^{k} U_{i}$ where $U_{i}:= V_{i}\backslash B$. Then $|U_{i}|\geq m-\varepsilon m(2r+1)$. By Lemma \ref{lem2.2}, $d(U_{i}, U_{j})\geq d(V_{i}, V_{j})-\varepsilon$ for distinct $i, j\in [k]$. Next, we shall prove that every $v\in U$ is 1-inner-reachable to linearly many vertices in $U$.

For each $p\in [k]$ and any vertex $v\in U_{p}$, we choose the minimum $q\in [\ell]$ such that $V_{p}\in V(\mathcal{K}_{q})$. Since $\mathcal{K}_{q}$ is a $K_{f(H)+1}$-embeddable structure, there exists a subgraph of $\mathcal{K}_{q}$, say $\mathcal{K}'_{q}$, which is a $K_{f(H)}$-embeddable structure such that $i_{\mathcal{K}'_{q}}(V_{p})=1$. Without loss of generality, we may assume $p=1$ and write $V(\mathcal{K}'_{q})=\{V_{1}, \dots, V_{a}, \dots, V_{a+b}\}$ for some integers $a, b\in \mathbb{N}$ and $a+2b=f(H)$. Thus for distinct $i, j\in [a+b]$, it follows from (2) and the fact $\varepsilon\ll \beta, \frac{1}{r}$ that every vertex $u\in U_{i}$ has at least $|N(u)\cap V_{j}|-|B\cap V_{j}|> d(V_{i}, V_{j})m-\varepsilon m-\varepsilon m(2r+1)\geq \frac{\beta}{2}m$ neighbors in $U_{j}$.

Recall that $v\in U_{1}$. We denote by $U_{1}^{\ast}$ the set of vertices $u\in U_{1}$ such that for every $j\in [2, a+b]$, $|N(u)\cap N(v)\cap U_{j}|\geq (\frac{\beta}{2})^{2}m$. The following claim would complete our proof because $|U_{1}^{\ast}|\geq(1-\varepsilon(2r+1))|U_{1}|\geq (1-\varepsilon(2r+1))^{2}m\geq \gamma_{1}n$, where $\gamma_{1}\ll \frac{1}{k}\ll \varepsilon$.


\begin{claim}
The vertex $v$ is $(H, \beta_{1}n, 1)$-reachable to every $u\in U^{\ast}_{1}$.
\end{claim}
To prove this, for every $u\in U_{1}^{\ast}$, we arbitrarily choose $N_{1}\subseteq U_{1}\backslash \{u, v\}$ with $|N_{1}|\geq (\frac{\beta}{2})^{2}m$ and write $N_j:=N(u)\cap N(v)\cap U_{j}$ for each $j\in [2, a+b]$. Then $|N_j|\geq (\frac{\beta}{2})^{2}m$ for each $j\in [a+b]$. For each $j\in [a+b]$, $N'_{j}$ comes from $N_j$ by deleting any $\beta_{1}n$ vertices.
Hence, $|N_j'|\geq |N_j|-\beta_{1}n\geq (\frac{\beta}{2})^{2}m-\beta_{1}n\geq \frac{\beta^{2}}{8}m$ since $\frac{1}{n}\ll \frac{1}{N}\ll\beta_{1}\ll\frac{1}{k}, \mu$ and $\beta=\frac{\mu}{10}$.
By Lemma \ref{lem2.2}, $(N'_{i}, N'_{j})$ is $\varepsilon'$-regular with $\varepsilon':=\max\left\{2\varepsilon, \frac{8\varepsilon}{\beta^{2}}\right\}=\frac{8\varepsilon}{\beta^{2}}$ and $d(N'_{i}, N'_{j})\geq d(V_{i}, V_{j})-\varepsilon$ for distinct $i, j\in [a+b]$.
Applying Corollary \ref{coro2.17} on $G[N'_{1}\cup \cdots\cup N'_{a+b}]$, we obtain a copy of $Q(a,b)$ which induces an independent set inside $N_1'$. Hence, by definition $v$ is $(H, \beta_{1}n, 1)$-reachable to $u$.
\end{proof}

\subsubsection{Proof of Lemma \ref{lem5.4}}

In the following, we shall use the latticed-based absorbing method developed by Han \cite{MR3632565} and begin with the following notion introduced by Keevash and Mycroft \cite{MR3290271}.
Let $G$ be an $n$-vertex graph. We will often work with a vertex partition $\mathcal{P}=\{V_{1}, \dots, V_{r}\}$ of $V(G)$ for some integer $r\geq 1$. For any subset $S\subseteq V(G)$, the \emph{index vector} of $S$ with respect to $\mathcal{P}$, denoted by $i_{\mathcal{P}}(S)$, is the vector in $\mathbb{Z}^{r}$ whose $i$th coordinate is the size of the intersections of $S$ with $V_{i}$ for each $i\in [r]$. For each $j\in [r]$, let $\textbf{u}_{j}\in \mathbb{Z}^{r}$ be the $j$th unit vector, i.e. $\textbf{u}_{j}$ has 1 on the $j$th coordinate and 0 on the other coordinates. A \emph{transferral} is a vector of the form $\textbf{u}_{i}-\textbf{u}_{j}$ for some distinct $i, j\in [r]$. A vector $\textbf{v}\in \mathbb{Z}^{r}$ is an \emph{$s$-vector} if all its coordinates are non-negative and their sum is $s$. Given $\mu>0$ and an $h$-vertex graph $H$, we say that an $h$-vector \textbf{v} is \emph{$(H, \mu)$-robust} if for any set $W$ of at most $\mu n$ vertices, there is a copy of $H$ in $G-W$ whose vertex set has an index vector equal to $\textbf{v}$. Let $I^{\mu}(\mathcal{P})$ be the set of all $(H, \mu)$-robust $h$-vectors and $L^{\mu}(\mathcal{P})$ be the lattice (i.e. the additive subgroup) generated by $I^{\mu}(\mathcal{P})$.

Here is a brief proof outline for Lemma~\ref{lem5.4}. In order to prove that $V(G)$ is closed, we adopt a less direct approach and build on the merging techniques developed in \cite{HMWY2021}.
We first partition $V(G)$ into a constant number of parts each of which is closed (see Lemma~\ref{lem5.8}).
Then we try to merge some of them into a larger (still closed) part by analyzing the graph structures.
Lemma~\ref{lem5.9} allows us to iteratively merge two distinct parts into a closed one, given the existence of a transferral. Therefore, the key step is to find a transferral (see Lemma~\ref{lem5.10}), where we shall use the regularity method and Corollary~\ref{coro2.17}. 
%

The following lemma can be used to construct a partition such that each part is closed.


\begin{lemma}[\cite{HMWY2021}, Lemma 3.10]\label{lem5.8}
For any positive constants $\gamma_{1}, \beta_{1}$, $h\in \mathbb{N}$ with $h\geq 3$ and an $h$-vertex graph $H$, there exist $\beta_{2}=\beta_{2}(\gamma_{1}, \beta_{1}, h)>0$ and $t_{2}\in \mathbb{N}$ such that the following holds for sufficiently large n. Let $G$ be an $n$-vertex graph such that every vertex in $V(G)$ is $(H, \beta_{1} n, 1)$-reachable to at least $\gamma_{1} n$ other vertices. Then there is a partition $\mathcal{P}=\{V_{1}, \dots, V_{p}\}$ of $V(G)$ with $p\leq \lceil\frac{1}{\gamma_{1}}\rceil$ such that for each $i\in [p]$, $V_{i}$ is $(H, \beta_{2} n, t_{2})$-closed and $|V_{i}|\geq \frac{\gamma_{1}}{2}n$.
\end{lemma}


\begin{lemma}[\cite{HMWY2021}, Lemma 4.4]\label{lem5.9}
Given any positive integers $h, t\in \mathbb{N}$ with $h\geq 3$, an $h$-vertex graph $H$ and constant $\beta>0$, the following holds for sufficiently large $n$. Let $G$ be an $n$-vertex graph with a partition $\mathcal{P}=\{V_{1}, \dots, V_{p}\}$ of $V(G)$ such that each $V_{i}$ is $(H, \beta n, t)$-closed. For distinct $i, j\in [p]$, if there exist two $h$-vectors $\textbf{s}, \textbf{t}\in I^{\beta}(\mathcal{P})$ such that $\textbf{s} - \textbf{t}=\textbf{u}_{i} - \textbf{u}_{j}$, then $V_{i}\cup V_{j}$ is $\left(H, \frac{\beta n}{2}, 2ht\right)$-closed.
\end{lemma}

Note that to invoke Lemma \ref{lem5.9}, we need the following result which provides a sufficient condition for the existence of a transferral.

\begin{lemma}[]\label{lem5.10}
Given $p, h\in \mathbb{N}$ with $h\geq 3$, an $h$-vertex graph $H$ and constants $\mu, \delta_{1}>0$, there exist $\alpha, \beta'>0$ such that the following holds for sufficiently large $n$. Let $G$ be an $n$-vertex graph with $\delta(G)\geq \max\left\{\left(1-\frac{2}{f(H)}+\mu\right)n, \left(\frac{1}{2}+\mu\right)n\right\}$, $\alpha(G)\leq \alpha n$, $\mathcal{P}=\{V_{1}, \dots, V_{p}\}$ be a partition of $V(G)$ with $|V_{i}|\geq \delta_{1} n$ for each $i\in [p]$. If $p\geq 2$, then there exist two $h$-vectors $\textbf{s}, \textbf{t}\in I^{\beta'}(\mathcal{P})$ such that $\textbf{s} - \textbf{t}=\textbf{u}_{i} - \textbf{u}_{j}$ for some distinct $i, j\in [p]$.
\end{lemma}

Now, we have collected all the tools needed for the proof of Lemma \ref{lem5.4}.

%
\begin{proof} [Proof of Lemma \ref{lem5.4}] Given $h\in \mathbb{N}$ with $h\geq 3$, an $h$-vertex graph $H$ and constants $\beta_{1}, \gamma_{1}, \mu>0$, we shall choose

\begin{center}
$\frac{1}{n}\ll\alpha\ll\beta, \frac{1}{t}\ll \beta_{2}, \frac{1}{t_{2}}\ll \beta_{1}, \gamma_{1}, \mu$.
\end{center}
Let $G$ be an $n$-vertex graph with $\delta(G)\geq \max\left\{\left(1-\frac{2}{f(H)}+\mu\right)n, \left(\frac{1}{2}+\mu\right)n\right\}$, $\alpha(G)\leq \alpha n$ and every vertex in $V(G)$ is $(H, \beta_{1}n, 1)$-reachable to at least $\gamma_{1}n$ other vertices. Applying Lemma \ref{lem5.8} on $G$, we obtain a partition $\mathcal{P}_{0}=\{V_{1}, \dots, V_{p}\}$ for some $p\leq \lceil\frac{1}{\gamma_{1}}\rceil$, where each $V_{i}$ is $(H, \beta_{2}n, t_{2})$-closed and $|V_{i}|\geq \frac{\gamma_{1}n}{2}$.

Let $\mathcal{P}'=\{U_{1}, \dots, U_{p'}\}$ be a vertex partition of $G$ with minimum $|\mathcal{P}'|$ such that $|U_{i}|\geq \frac{\gamma_{1}n}{2}$ and $U_{i}$ is $(H, \beta n, t)$-closed. We claim that $p'=1$. If $p'\geq 2$, then by Lemma \ref{lem5.9} and Lemma \ref{lem5.10}, there exist two distinct vertex parts $U_{i}$ and $U_{j}$ for distinct $i, j\in [p']$ such that $U_{i}\cup U_{j}$ is $(H, \beta'n, t')$-closed for some $\beta'$ and $t'$. By taking $U_{i}\cup U_{j}$ as a new part in partition and renaming all the parts if necessary, we get a partition $\mathcal{P}''$ with $|\mathcal{P}''|<|\mathcal{P}'|$, which contradicts the minimality of $|\mathcal{P}'|$. Hence, $V(G)$ is $(H, \beta n, t)$-closed.
\end{proof}

Next, we give a proof of Lemma \ref{lem5.10}. In order to prove Lemma \ref{lem5.10}, we use the regularity lemma (Lemma \ref{lem2.3}) and an embedding result (Claim \ref{claim5.13}). In particular, such an embedding result allows us to construct vertex-disjoint copies of $H$ with different index vectors, which can be used to show the existence of a transferral. This roughly reduces the problem to finding in the reduced graph a crossing $K_{f(H)+1}$-embeddable structure, which will be made precise later.

\begin{proof}[Proof of Lemma \ref{lem5.10}] Given $p, h, t\in \mathbb{N}$, an $h$-vertex graph $H$ and positive constants $\mu, \delta_{1}$, we shall choose
\begin{center}
$\frac{1}{n}\ll\alpha\ll\frac{1}{N}\ll\beta'\ll\frac{1}{k}\ll\varepsilon\ll\mu, \delta_{1}$.
\end{center}
Let $\beta=\frac{\mu}{10}$, $r:=ar(H)$, $\mathcal{T}=\{T_{1}, \dots, T_{r}\}$ be an acyclic partition of $H$, $\mathcal{P}=\{V_{1}, \dots, V_{p}\}$ be a vertex partition of $V(G)$ with $|V_{i}|\geq \delta_{1} n$ for each $i\in [p]$. Anchoring at the current vertex partition of $V(G)$, we apply Lemma \ref{lem2.3} with $\varepsilon, \beta>0$ and refine the current partition.
After refinement, we denote the $\varepsilon$-regular partition by $\mathcal{P}'=\{V_{0}, V_{1, 1}, \dots, V_{1, s_{1}}, \dots, V_{p, 1}, \dots, V_{p, s_{p}}\}$ where $V_{i, j}\subseteq V_{i}$ and $s_{i}\in \mathbb{N}$ for each $i\in [p]$, $j\in [s_{i}]$. Let $R:=R_{\beta, \varepsilon}$ be the reduced graph with $|V(R)|=k$ and $\mathcal{V}_{i}:=\{V_{i, 1}, \dots, V_{i, s_{i}}\}$ be a vertex subset of $V(R)$ for each $i\in[p]$.
For every cluster $V_{i, j}$, $D_{i, j}$ denotes the double-edge neighborhood of $V_{i, j}$.
It holds that $2|D_{i, j}|+(k-|D_{i, j}|) \geq \delta(R)$ which means that

\begin{equation}\label{eq2}
|D_{i, j}|\geq \delta(R)-k,
\end{equation}
for each $V_{i, j}\in V(R)$.

We call a subgraph $\mathcal{K}\subseteq R$ \emph{crossing} with respect to the partition $\mathcal{P}$ if $V(\mathcal{K})\cap \mathcal{V}_{i}\neq \emptyset$ and $V(\mathcal{K})\cap \mathcal{V}_{j}\neq \emptyset$ for some distinct $i, j\in [p]$. A \emph{double-edged} $\mathcal{K}$ has every two adjacent vertices connected by a double-edge. We use $K^{=}_{3}$ to denote the triangle which contains exactly one double-edge and write $f:=f(H)$ throughout this proof.

The following claim provides a sufficient condition for the existence of a transferral. Its proof is postponed to the end of this subsection.

\begin{claim}\label{claim5.13}
If there is a crossing $K_{f+1}$-embeddable structure $\mathcal{K}$ in $R$, then there exist two $h$-vectors $\textbf{s}, \textbf{t}\in I^{\beta'}(\mathcal{P})$ such that $\textbf{s} - \textbf{t}=\textbf{u}_{i} - \textbf{u}_{j}$ for distinct integers $i$ and $j$.
\end{claim}

Thus, we may assume that there is no crossing $K_{f+1}$-embeddable structure. Note that if there exists a crossing double-edge between $\mathcal{V}_{i}$ and $\mathcal{V}_{j}$ for some distinct $i, j\in [p]$, then by Lemma \ref{lem5.11} the double-edge is contained in a $K_{f+1}$-embeddable structure which is crossing, a contradiction. So we may further assume that there is no crossing double-edge in $R$. In this case, we shall find a crossing $K_{f+1}$-embeddable structure and this gives a final contradiction.

In the following, we assume $|\mathcal{V}_{i}|\leq |\mathcal{V}_{i+1}|$ for each $i\in [p-1]$ and $x:=|\mathcal{V}_{1}|$ for some integer $x\leq \frac{k}{2}$. To find a crossing $K_{f+1}$-embeddable structure without any crossing double-edge, we divide the proof into the following four cases.

Assume $f\geq 8$. By Fact \ref{fact2.5}, we have $\delta(R)\geq \left(\frac{3}{2}+\mu\right)k$. Hence for $V_{1, 1}\in \mathcal{V}_{1}$, by (\ref{eq2}) it holds that

\begin{center}
  $\frac{k}{2}\geq x=|\mathcal{V}_{1}|\geq |D_{1, 1}|\geq \delta(R)-k\geq \left(\frac{1}{2}+\mu\right)k,$
  \end{center}
a contradiction.

Assume $f=7, 6$. By Fact \ref{fact2.5}, we have $\delta(R)\geq 2\left(1-\frac{2}{f}+\frac{\mu}{2}\right)k$. Let $V_{1, i}V_{1, j}$ be a double-edge in $R[\mathcal{V}_{1}]$ for distinct $i, j\in [s_{1}]$, $\mathcal{V}':=(N_{R}(V_{1, i})\cap N_{R}(V_{1, j}))\backslash\mathcal{V}_{1}$ and $y:=|\mathcal{V}'|$. Since there is no crossing double-edge, it holds that
\begin{align}\label{eq3}
        y & \geq 2\left[2\left(1-\frac{2}{f}+\frac{\mu}{2}\right)k-2x\right]-(k-x)\nonumber\\
         & =\left(3-\frac{8}{f}+2\mu\right)k-3x\nonumber\\
         & \geq \left(2-\frac{8}{f}+2\mu\right)k-x.
\end{align}
      For any $V_{w, \ell}\in \mathcal{V}'$, it holds that \begin{align*}
      |N_{R[\mathcal{V}']}(V_{w, \ell})|& \geq  \nonumber \frac{1}{2}\left[\delta(R)-2|V(R)\backslash(\mathcal{V}_{1}\cup\mathcal{V}')|-|\mathcal{V}_{1}|\right]\\ \nonumber
                                  & =\frac{1}{2}\left[2\left(1-\frac{2}{f}+\frac{\mu}{2}\right)k-2(k-x-y)-x\right] \\ \nonumber
                                  & =\frac{1}{2}\left[\left(-\frac{4}{f}+\mu\right)k+x+2y\right]. \nonumber
                               \end{align*}
If there is a copy of $K_{f-3}$ in $R[\mathcal{V}']$, then this copy combined with $V_{1, i}V_{1, j}$ forms a crossing copy of $K_{f+1}$-embeddable structure. So by Tur\'{a}n's theorem, we must have
\begin{align*}
& \frac{1}{2}\left[\left(-\frac{4}{f}+\mu\right)k+x+2y\right] <\frac{f-5}{f-4}y; \\
& \Longleftrightarrow \frac{1}{f-4}y <\frac{1}{2}\left(\frac{4}{f}-\mu\right)k-\frac{1}{2}x;\\
& \Longleftrightarrow y<\left(2-\frac{8}{f}-\frac{f-4}{2}\mu\right)k-\frac{f-4}{2}x.
\end{align*}
However, this contradicts \eqref{eq3} and $f\geq 6$.

%
%

\begin{figure}[htbp]\label{Figure1}
\centering
\includegraphics[scale=0.8]{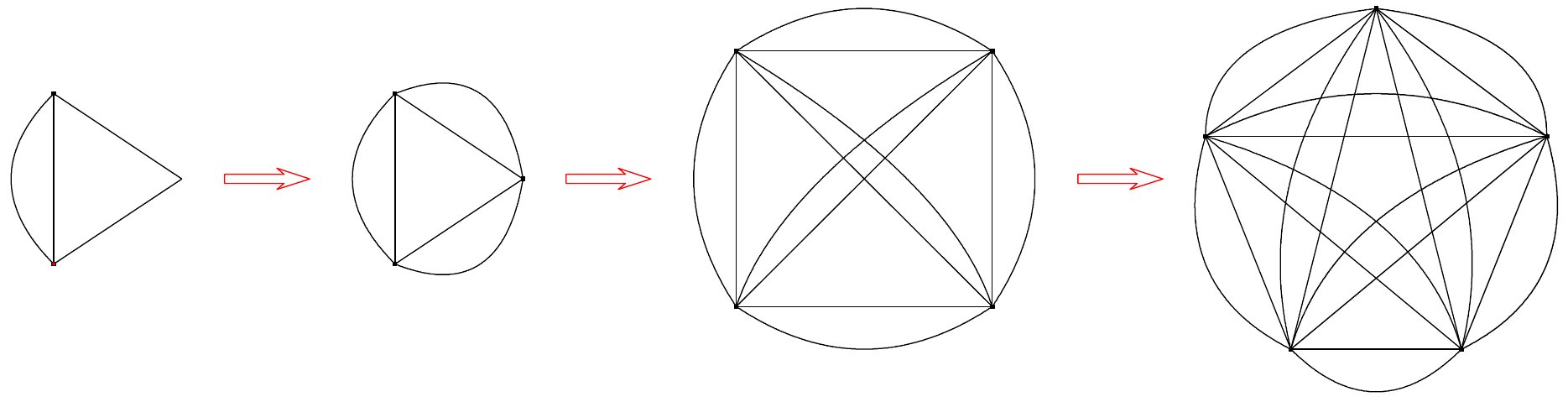}
\caption{Illustration of a graph ordering}
\label{L1}
\end{figure}
%
%

Assume $f=5$. For any cluster $V_{1, i}\in \mathcal{V}_{1}$ for $i\in [s_{1}]$, it holds that $|N_{R[\mathcal{V}_{1}]}(V_{1, i})|\geq \frac{\frac{6}{5}k-(k-x)}{2}=\frac{x}{2}+\frac{k}{10}$. Thus every double-edge $V_{1, i}V_{1, j}$ is contained in a copy of $K^{=}_{3}$ in $R[\mathcal{V}_{1}]$. Based on this, we shall first show that there is no double-edged $K_{5}$ in $R[\mathcal{V}_{1}]$. Suppose for a contradiction that there exists a double-edged $K_{5}$ in $R[\mathcal{V}_{1}]$, whose vertex set is denoted by $\{V_{1,1}, V_{1,2}, \dots, V_{1,5}\}$. Then we claim that there exists a crossing $K_{7}$-embeddable structure which consists one cluster outside $\mathcal{V}_{1}$ and three clusters in the double-edged $K_{5}$. Indeed, we calculate the edges between $\{V_{1,1}, V_{1,2}, \dots, V_{1,5}\}$ and $V(R)\backslash \mathcal{V}_{1}$. We aim to show $e(\{V_{1,1}, V_{1,2}, \dots, V_{1,5}\}, V(R)\backslash \mathcal{V}_{1})>2(k-x)$, which would imply the existence of one cluster in $V(R)\backslash \mathcal{V}_{1}$ with at least three neighbors in the double-edged $K_{5}$. Note that this gives a $K_{7}$-embeddable structure. It suffices to show
\begin{align*}
& \sum_{i=1}^{5}d(V_{1, i}) >\sum_{i=1}^{5}d_{R[\mathcal{V}_{1}]}(V_{1, i})+2(k-x);\\
&\Longleftarrow 5\times\left(\frac{6}{5}+\mu\right)k>5\times 2x+2(k-x);\\
&\Longleftrightarrow x<\frac{4+5\mu}{8}k.
\end{align*}
As $x=|\mathcal{V}_{1}|\leq \frac{k}{2}$, we are done. In the following, we may further assume that there is no double-edged $K_{5}$ in $R[\mathcal{V}_{1}]$.

Now we shall define an ordering of the graphs in $\{K^{=}_{3}, \mbox{double-edged}\:K_{3}, \mbox{double-edged}\:K_{4}, \\\mbox{double-edged}\:K_{5}\}$ as illustrated in Figure 1. Let $S$ be a maximal element in the chain which is a subgraph of $R[\mathcal{V}_{1}]$. Without loss of generality, let $V(S)=\{V_{1, 1}, \dots, V_{1, a}\}$. Then $a\in \{3, 4\}$. Observe that $\bigcap_{j\in [a]}D_{1,j}=\emptyset$. Then we have $e(V(S), V(R)\backslash \mathcal{V}_{1})\geq \sum_{i=1}^{a}d(V_{1, i})-(2a-1)(x-a)-2a^{2}=\sum_{i=1}^{a}d(V_{1, i})-(2a-1)x-a$. Since $x\leq \frac{k}{2}$, no matter $a=3$ or $4$, we always have

\begin{center}
$x <\frac{(6a-10)k-5a}{10a-15}$\text{, that is, }$\frac{6}{5}ak-(2a-1)x-a>2(k-x)$,

\end{center}
which implies that $e(V(S), V(R)\backslash \mathcal{V}_{1})\geq \sum_{i=1}^{a}d(V_{1, i})-(2a-1)x-a >2(k-x).$ This means that there is a cluster in $V(R)\setminus \mathcal{V}_{1}$ which has at least three neighbors in $V(S)$, giving a crossing $K_{6}$-embeddable structure consisting of one cluster in $V(R)\backslash \mathcal{V}_{1}$ and three clusters in $V(S)$. So we are done.

Finally we have $f=4, 3, 2$. We assume that there is no crossing $K_{3}^{=}$ as otherwise we are done. By Fact \ref{fact2.5}, we have $\delta(R)\geq (1+\mu)k$ and without loss of generality, we may assume $V_{1,1}V_{2,1}$ is a crossing single-edge. Then $N(V_{1, 1})\cap D_{2, 1}=\emptyset$ and $N(V_{2, 1})\cap D_{1, 1}=\emptyset$. We have $d(V_{1, 1})\leq k-|D_{1, 1}|-|D_{2, 1}|+2|D_{1, 1}|$ and $d(V_{2, 1})\leq k-|D_{1, 1}|-|D_{2, 1}|+2|D_{2, 1}|$ which yields $(2+2\mu)k\leq2\delta(R)\leq d(V_{1, 1})+d(V_{2, 1})\leq2k$, a contradiction.
\end{proof}

Now we prove Claim \ref{claim5.13}.

\begin{proof} [Proof of Claim \ref{claim5.13}]
Recall that $V(R)=\{V_{1, 1}, \dots, V_{1, s_{1}}, \dots, V_{p, 1}, \dots, V_{p, s_{p}}\}$ with $|V(R)|=k$ and $\mathcal{V}_{i}=\{V_{i, 1}, \dots, V_{i, s_{i}}\}$, for each $i\in[p]$.
Without loss of generality, we assume that $\mathcal{K}$ is a crossing $K_{f+1}$-embeddable structure in $R$ such that $V(\mathcal{K})=\{U_1, U_2, \dots, U_{a+b}\}$ for distinct clusters $U_1, U_2, \dots, U_{a+b}\in V(R)$ such that $U_i\in\mathcal{V}_i$ for $i\in[2]$, where $a+2b=f+1$.

If there exists a cluster in $V(\mathcal{K})$, say $U_{q}\in \mathcal{V}_{\ell}$ for some $q\in [a+b]$ and $\ell\in[p]$, such that $i_{\mathcal{K}}(U_q)=1$, then $\mathcal{K}-\{U_{q}\}=:\mathcal{K}'$ is a $K_{f}$-embeddable structure. Now we shall find two $h$-vectors $\textbf{s}, \textbf{t}\in I^{\beta'}(\mathcal{P})$ as required. Let $U'_{i}$ be a subset of $U_{i}$ for every $i\in[a+b]$, by deleting any $\beta'n$ vertices. Since $(U_{q}, U_{i})$ is an $(\varepsilon,\beta)$-regular pair for each $i\in[a+b]\setminus \{q\}$ and $\frac{1}{n}\ll\frac{1}{N}\ll \beta'\ll\frac{1}{k}\ll\beta$, we pick a vertex $v\in U'_{q}$ such that $|N(v)\cap U'_{i}|\geq (\beta-\varepsilon)m-\beta'n\ge\frac{\beta}{2} m\ge N$. Write $Y_i=N(v)\cap U'_{i}$ for each $i\in[a+b]\setminus \{q\}$. By Lemma \ref{lem2.2}, $(Y_{i}, Y_{j})$ is $\varepsilon'$-regular with $\varepsilon':=\max\left\{2\varepsilon, \frac{2\varepsilon}{\beta}\right\}=\frac{2\varepsilon}{\beta}$ for distinct $i, j\in[a+b]\setminus \{q\}$ and $d(Y_{i}, Y_{j})\geq d(U_{i}, U_{j})-\varepsilon$. Applying Corollary \ref{coro2.17} on $G[\cup_{i\neq q}Y_{i}]$, there exists a copy of $Q(a-1,b)$. Recall that $U_i\in\mathcal{V}_i$ for $i\in[2]$. Thus there exists $i_0\in[2]$ such that $\ell\neq i_0$. Let $H_{1}$ be a copy of $H$ in $Q(a-1,b)$ such that $V(H_{1})\cap U_{i_{0}}\neq\emptyset$. Then we replace any vertex in $V(H_{1})\cap U_{i_0}$ with $v$ and get another copy of $H$, say $H_2$. The two copies $H_1,H_2$ respectively give two $h$-vectors $\textbf{s}, \textbf{t}\in I^{\beta'}(\mathcal{P})$ such that $\textbf{s} - \textbf{t}=\textbf{u}_{\ell} - \textbf{u}_{i_0}$.


Now it remains to deal with the case that every $U_{i}\in V(\mathcal{K})$ satisfies $i_{\mathcal{K}}(U_i)=2$, $i\in [a+b]$. Then $a=0, f=2b-1$. Here we obtain that $H\in\widetilde{\mathcal{H}}$ and $b=ar(H)$.
Let $\mathcal{T}=\{T_{1}, \dots, T_{b}\}$ be an acyclic partition of $H$ such that $H[T_1]$ is an $s$-independent set for some $s\in\mathbb{N}$ and $|T_i|=2s,i\in[2,b]$.
Let $F_{1}=K_{1, 2s-1}$ and $F_{i}=2H[T_{i}]$ for each $i\in [2, b]$.
For each $U_{i}$, let $U'_{i}$ be obtained from $U_{i}$ by deleting any $\beta'n$ vertices. Since $\frac{1}{n}\ll\frac{1}{N}\ll \beta'\ll\frac{1}{k}$, $|U'_{i}|\geq m-\beta'n\geq \frac{m}{2}\ge N$ for each $i\in [b]$. By Lemma \ref{lem2.2}, $d(U'_{i}, U'_{j})\geq d(U_{i}, U_{j})-\varepsilon\ge\frac{1}{2}+\frac{\beta}{2}$ and $(U'_{i}, U'_{j})$ is $\varepsilon'$-regular with $\varepsilon':=\max\{2\varepsilon, \frac{\varepsilon|U_{i}|}{|U'_{i}|}\}=2\varepsilon$ for distinct $i,j\in [b]$.  Applying Lemma \ref{lem2.17} on $G[U'_{1}\cup \cdots\cup U'_{b}]$, we obtain a copy of $Q(0, b, 2s, F_{1}, \dots, F_b)$, say $Q$, such that $V(Q)=X_1\cup X_2\cup\cdots\cup X_b$ and $X_i\subseteq U_i', i\in[b]$. Note that $Q[X_1]$ induces a copy of $K_{1,2s-1}$ whose center is denoted by $v^*$.
It is easy to derive that $Q$ contains a copy of $H$, say $H_{1}$, such that $V(H_{1})\cap X_{1}$ are leaves of the $K_{1,2s-1}$, $|V(H_{1})\cap U_{i}|=2|V(H_{1})\cap U_{1}|=2s$ for every $i\in [2, b]$. By removing from $H_{1}$ any vertex in $V(H_{1})\cap U_{2}$ and adding the center $v^*$, we obtain another copy of $H$, say $H_{2}$ such that $|V(H_2)\cap U_{1}|=s+1,|V(H_2)\cap U_{2}|=2s-1$ and $|V(H_2)\cap U_{i}|=2s, i\in[3,b]$. So $H_{1}$ and $H_{2}$ provide two $h$-vectors $\textbf{s}, \textbf{t}\in I^{\beta'}(\mathcal{P})$ such that $\textbf{s} - \textbf{t}=\textbf{u}_{2} - \textbf{u}_{1}$.
\end{proof}
\subsection{Proof of Lemma \ref{lem5.7}}
This subsection is devoted to the proof of Lemma \ref{lem5.7} which states that the given minimum degree and independence number suffice to guarantee that every vertex is in a copy of $H$ while excluding any vertex set $W$ of size $o(n)$. To achieve this goal, we need the following result which investigates the structure around every vertex in the original graph $G$.



\begin{lemma}[\cite{MR4193066}, Lemma 3.11]\label{lem5.12}
Given $r\in \mathbb{N}$ with $r\geq 4$ and constants $\varepsilon, \beta, \mu$ with $0<\varepsilon, \beta\leq \frac{\mu}{10}$, the following holds for sufficiently large $n$. Let $G$ be an $n$-vertex graph with $\delta(G)\geq \left(1-\frac{2}{r}+\mu\right)n$, $\mathcal{P}=\{V_{0}, \dots, V_{k}\}$ be an $\varepsilon$-regular partition for some integer $k$ and $R$ be a reduced multigraph with multiplicity $2$. Fix a vertex $v$ in $G$, and let $Q_{v}$ be the set of clusters $V_{i}\in V(R)$ such that $|N_{V_{i}}(v)|\geq \beta|V_{i}|$. Then there exists a multi-embedding of $K_{r}$ into $R$ embedding at most one vertex in $V(R)\backslash Q_{v}$.
\end{lemma}

\begin{proof} [Proof of Lemma \ref{lem5.7}] Given $h\in \mathbb{N}$ with $h\geq 3$, an $h$-vertex graph $H$ and constant $\mu$, we choose
\begin{center}
$\frac{1}{n}\ll \frac{1}{N}\ll \alpha\ll\frac{1}{k}\ll \varepsilon\ll \mu$.
\end{center}
Let $\beta=\frac{\mu}{10}$, $G$ be an $n$-vertex graph with $\delta(G)\geq \max\left\{\left(1-\frac{2}{f(H)}+\mu\right)n, \left(\frac{1}{2}+\mu\right)n\right\}$, $\alpha(G)\leq \alpha n$, $W\subseteq V(G)$ with $|W|\leq \frac{\mu}{2}n$ and $G_{1}:= G-W$. Then we have
\begin{center}
$\delta(G_{1})\geq \max\left\{\left(1-\frac{2}{f(H)}+\mu\right)n, \left(\frac{1}{2}+\mu\right)n\right\}-|W|\geq \max\left\{\left(1-\frac{2}{f(H)}+\frac{\mu}{2}\right)n, \left(\frac{1}{2}+\frac{\mu}{2}\right)n\right\}$.
\end{center}
Applying Lemma \ref{lem2.3} to $G_{1}$ with $\varepsilon, \beta>0$, we obtain an $(\varepsilon,\beta)$-regular partition $\mathcal{P}=\{V_{0}, V_{1}, \dots, V_{k}\}$ of $G_{1}$.
Let $R:=R_{\beta, \varepsilon}$ be a reduced multigraph for the partition $\mathcal{P}$ with multiplicity $2$, $|V(R)|=k$ and $m:=|V_{i}|$ for each $i\in [k]$. By Fact \ref{fact2.5}, we have $\delta(R)\geq 2\left(1-\frac{2}{f(H)}+\frac{\mu}{2}\right)k$ when $f(H)\geq 4$ and $\delta(R)\geq 2\left(\frac{1}{2}+\frac{\mu}{2}\right)k$ when $f(H)=2, 3$. Applying Lemma \ref{lem5.12} with $r=\max\{f(H), 4\}$ where $\frac{\mu}{2}$ plays the role of $\mu$, for any vertex $v\in V(G_{1})$, we obtain a $K_{f(H)}$-embeddable structure $\mathcal{K}$ such that $|V(\mathcal{K})\backslash Q_{v}|\le1$ where $Q_{v}$ is as defined in Lemma \ref{lem5.12}. Assume $V(\mathcal{K})=\{V_{1}, \dots, V_{\ell}\}$.

If $V(\mathcal{K})\backslash Q_{v}=\emptyset$, then $\mathcal{K}$ is a subgraph of $R[Q_{v}]$. Let $S_{i}:= N(v)\cap V_{i}$ for each $i\in [\ell]$. Then by the definition of $Q_{v}$, we have $|S_{i}|\geq \beta m\geq N$ for each $i\in [\ell]$ since $\frac{1}{n}\ll\frac{1}{N}\ll\frac{1}{k}\ll\beta$. By Lemma \ref{lem2.2}, $(S_{i}, S_{j})$ is $\varepsilon'$-regular with $\varepsilon':=\max\left\{2\varepsilon, \frac{\varepsilon}{\beta}\right\}=\frac{\varepsilon}{\beta}$ and $d(S_{i}, S_{j})\geq d(V_{i}, V_{j})-\varepsilon$ for distinct $i, j\in [\ell]$. Corollary \ref{coro2.17} applied on $G[S_{1}\cup \cdots\cup S_{\ell}]$ with $a=2\ell-f(H)$ and $b=f(H)-\ell$ gives a copy $Q$ of $Q(a, b)$. Note that $V(Q)\subseteq N(v)$. Hence, $v$ is in a copy of $H$.

%

Now assume $V(\mathcal{K})\backslash Q_{v}=\{V_{\ell}\}$, and thus $i_{\mathcal{K}}(V_{\ell})=1$. Similarly, we can choose subsets $S_{i}\subseteq V_{i}\cap N(v)$ as above for each $i\in [\ell-1]$. Applying Corollary \ref{coro2.17} on $G[S_{1}\cup S_{2}\cup \cdots\cup S_{\ell-1}\cup V_{\ell}]$ with $a=2\ell-f(H)$, $b=f(H)-\ell$ and the fact that $i_{\mathcal{K}}(V_{\ell})=1$, we obtain a copy $Q$ of $Q(a, b)$ such that $Q[V_{\ell}]$ is an independent set. Replacing any vertex in the independent set with $v$, we conclude that $v$ is in a copy of $H$.
\end{proof}

%
%
%
%
%

%
%
%
%
%

\bibliographystyle{abbrv}
\bibliography{ref}

\end{document}